\documentclass[12pt,a4paper]{amsart}
\usepackage[latin1]{inputenc}
\usepackage[english]{babel}
\usepackage{graphicx}
\usepackage{amsmath,amssymb,amsthm}
\usepackage{yfonts}

\theoremstyle{plain}
\newtheorem{theorem}{Theorem}[section]
\newtheorem{lemma}[theorem]{Lemma}
\newtheorem{corollary}[theorem]{Corollary}
\newtheorem{prop}[theorem]{Proposition}

\theoremstyle{definition}

\newtheorem{remark}[theorem]{Remark}

    \newtheoremstyle{TheoremNum}
        {\topsep}{\topsep}              
        {\itshape}                      
        {}                              
        {\bfseries}                     
        {.}                             
        { }                             
        {\thmname{#1}\thmnote{ \bfseries #3}}
    \theoremstyle{TheoremNum}
    \newtheorem{theoremn}{Theorem}

    \newtheoremstyle{CorollaryNum}
        {\topsep}{\topsep}              
        {\itshape}                      
        {}                              
        {\bfseries}                     
        {.}                             
        { }                             
        {\thmname{#1}\thmnote{ \bfseries #3}}
    \theoremstyle{CorollaryNum}
    \newtheorem{corollaryn}{Corollary}

\newtheoremstyle{PropositionNum}
        {\topsep}{\topsep}              
        {\itshape}                      
        {}                              
        {\bfseries}                     
        {.}                             
        { }                             
        {\thmname{#1}\thmnote{ \bfseries #3}}
    \theoremstyle{PropositionNum}
    \newtheorem{propn}{Proposition}

    \newtheoremstyle{ConjectureNum}
        {\topsep}{\topsep}              
        {\itshape}                      
        {}                              
        {\bfseries}                     
        {.}                             
        { }                             
        {\thmname{#1}\thmnote{ \bfseries #3}}
    \theoremstyle{TheoremNum}

 \renewcommand{\leq}{\leqslant}
\renewcommand{\geq}{\geqslant}

\newcounter{equi1}

\newcommand{\bea}{\begin{eqnarray*}}
\newcommand{\eea}{\end{eqnarray*}}
\newcommand{\beq}{\begin{equation}}
\newcommand{\eeq}{\end{equation}}
\newcommand{\begsta}{\begin{statements}}
\def\endsta{\end{statements}}
\newcommand{\begaeq}{\begin{aequivalenz}}
\def\endaeq{\end{aequivalenz}}
\begin{document}

\title[Multiplicative funtions] {Correlations of multiplicative functions and applications}



\author[Klurman]{Oleksiy Klurman}

\address{D\'epartment de Math\'ematiques et de Statistique,
Universit\'e de Montr\'eal, CP 6128 succ. Centre-Ville, Montr\'eal QC H3C 3J7, Canada
Canada}

\address{Department of Mathematics, University College London, Gower Street, London, WC1E 6BT, UK}

 \email{\texttt{lklurman@gmail.com}}
\begin{abstract}
We give an asymptotic formula for correlations 
\[
\sum_{n\le x}f_1(P_1(n))f_2(P_2(n))\cdot \dots \cdot f_m(P_m(n))\] where $f\dots,f_m$ are bounded ``pretentious" multiplicative functions, under certain natural hypotheses.  We then deduce several desirable consequences:\ First, we characterize all multiplicative functions $f:\mathbb{N}\to\{-1,+1\}$ with bounded partial sums. This answers a question of Erd\H{o}s from $1957$ in the form conjectured by Tao. Second, we show that if the average of the first divided difference of multiplicative function is zero, then either $f(n)=n^s$ for $\operatorname{Re}(s)<1$ or $|f(n)|$ is small on average. This settles an old conjecture of K\'atai. Third, we apply our theorem to count the number of representations of $n=a+b$ where $a,b$ belong to some multiplicative subsets of $\mathbb{N}.$ This gives a new "circle method-free" proof of the result of Br{\"u}dern. \end{abstract}

\maketitle
\begin{section}{Introduction}
Let $\mathbb{U}$ denote the unit disc, and let  $\mathbb{T}$ be the unit circle. It is of current interest in analytic number theory to understand the correlations
\[
\sum_{n\le x}f_1(P_1(n))f_2(P_2(n))\cdot \dots \cdot f_m(P_m(n))\]
for arbitrary multiplicative functions $f_1,\dots,f_m:\mathbb{N}\to\mathbb{U}$, and arbitrary polynomials $P_1,\dots, P_m\in\mathbb{Z}[x].$ For example, Chowla's conjecture that for any distinct natural numbers $h_1,\dots h_k$
\[\sum_{n\le x}\lambda(n+h_1)\dots\lambda(n+h_k)=o(x)\]
where $\lambda(n)$ is a Liouville function.
These problems are still widely open in general, though spectacular progress has been made recently due to the breakthrough of Matom{\"a}ki and Radziwi{\l}{\l}~\cite{MR5} and subsequent work of Matom{\"a}ki, Radziwi{\l}{\l} and Tao~\cite{MR6}. In particular,  this led Tao~\cite{MR4} to establish a weighted version of Chowla's conjecture in the form
\[\sum_{n\le x}\frac{\lambda(n)\lambda(n+h)}{n}=o(\log x)\]
for all $h\ge 1.$ Combining this with ideas from the Polymath5 project, and a new ``entropy decrement argument",  led to the resolution of the Erd\H{o}s Discrepancy Problem. 

Following Granville and Soundararajan~\cite{MR2276774}, we define the ``distance" between two multiplicative functions $f,g:\mathbb{N}\to\mathbb{U}$
\[\mathbb{D}(f,g;y;x)=\left(\sum_{y\le p\le x}\frac{1-\operatorname{Re}{(f(p)\overline{g(p)}})}{p}\right)^{\frac{1}{2}},\]
and $\mathbb{D}(f,g;x):=\mathbb{D}(f,g;1;x)$.
The crucial feature of this ``distance" is that it satisfies the triangle inequality
\[\mathbb{D}(f,g;y;x)+\mathbb{D}(g,h;y;x)\ge \mathbb{D}(f,h;y;x)\]
for any multiplicative functions $f,g,h$ bounded by $1$.

 Hal\'asz's theorem~\cite{MR0319930},~\cite{MR0369292} implies Wirsing's Theorem that
for multiplicative $f:\mathbb{N}\to [-1,1]$, the mean value satisfies a decomposition into local factors,
\begin{equation} \label{Meanf}
\frac{1}{x}\sum_{n\le x}f(n) = \prod_p M_p(f) +o_{x\to \infty} (1)
\end{equation}
where we define the multiplicative function $f_p$ for each prime $p$ to be 
\begin{equation}\label{localfactor}
 f_{p}(q^k) = \begin{cases} f(q^k), & \mbox{if } q=p \\ 1, & \mbox{if } q\ne p,\end{cases}
 \end{equation}
for all $k\geq 1$, and 
$$
M_p(f):= \lim_{x\to \infty} \frac 1x \sum_{n\le x}f_p(n) = \left(1-\frac{1}{p}\right) \sum_{k\ge 0}\frac{f(p^k)}{p^k}.
$$
This last equality, evaluating $M_p(f)$, is an easy exercise.   Substituting this into \eqref{Meanf} one finds that the mean value there is $\asymp \exp( -\mathbb{D}(f,1;\infty))^2$, and so  is non-zero if and only if 
$\mathbb{D}(f,1;\infty)<\infty$ and each $M_p(f)\ne 0$.  Moreover, using our explicit evaluation of $M_p(f)$, we see that $M_p(f)=0$ if and only if $p=2$ and $f(2^k)=-1$ for all $k\geq 1$. We also note that one can truncate the product in \eqref{Meanf} to the primes $p\leq x$, and retain the same qualitative result.

\bigskip

\noindent \textbf{1.1.\ Mean values of multiplicative functions acting on polynomials}.
Our first goal is to prove the analogy to \eqref{Meanf} for the mean value of $f(P(n))$ for any given polynomial $P(x)\in \mathbb Z[x]$. This is not difficult for linear polynomials $P$ but, as the following example shows, it is not so straightforward for higher degree polynomials:
 


\begin{prop}\label{dependence}
There exists a multiplicative function $f:\mathbb{N}\to [-1,1]$ such that $\mathbb{D}(1,f;x)=2\log\log x+O(1)$ for all $x\ge 2$ and 
\[\limsup_{x\to\infty}\left|\frac{1}{x}\sum_{n\le x}f(n^2+1)\right|\ge \frac{1}{2}+o(1).\]
\end{prop}

In the proof of Proposition~\ref{dependence}, the choice of $f(p)$ for certain  primes $p\ge x$ have a significant impact on the mean value of $f(n^2+1)$ up to $x$.
In order to tame this effect we introduce the set
\[N_P(x) = \{ p^k, p\ge x \mid \exists n \le x,\ p^k\vert|P(n)  \}\]
for any given $P\in\mathbb{Z}[x],$ and
 modify the ``distance" to
 \[
 \mathbb{D}_P(f,g;y;x)=\left(\sum_{y\le p\le x}\frac{1-\operatorname{Re}{(f(p)\overline{g(p)}})}{p}+\sum_{p^k\in N_P(x)}\frac{1-\operatorname{Re}{(f(p^k)\overline{g(p^k)}})}{x}\right)^{\frac{1}{2}}.
 \]
Moreover we define
 \[
 M_p(f(P))=\lim_{x\to\infty}\frac{1}{x}\sum_{n\le x}f_{p}(P(n)) ,
 \]
and one easily shows that
\[
M_p(f(P))=\sum_{k\ge 0}f(p^k)\left(\frac{\omega_P(p^k)}{p^k}-\frac{\omega_P(p^{k+1})}{p^{k+1}}\right),
\]
where $\omega_P(m):=\#\{n \pmod{m}:\ P(n)\equiv 0\pmod{m}\}$ for every integer $m$ (and note that $\omega_P(.)$ is a multiplicative function by the Chinese Remainder Theorem). We establish the following analogy to
\eqref{Meanf}:

 \begin{corollary}\label{nair}
Let $f:\mathbb{N}\to\mathbb{U}$ be a multiplicative function and let $P(x)\in\mathbb{Z}[x]$ be a polynomial. Then
\[
\frac{1}{x}\sum_{ n\le x}f(P(n))=\prod_{p\le x}M_p(f(P))+O\left(\mathbb{D}_{P}(1,f;\log x;x)+\frac{1}{\log \log x}\right).
\]
\end{corollary}

This implies that if $\mathbb{D}(1,f;x)<\infty$ and 
\[\sum_{p^k\in N_P(x)}1-\operatorname{Re}(f(p^k))=o(x)\]
 then 
\[
\frac{1}{x}\sum_{ n\le x}f(P(n))=\prod_{p\le x}M_p(f(P))+o(1)=\prod_{p\ge 1}M_p(f(P))+o(1).
\]

\bigskip

\noindent \textbf{1.2.\ Mean values of correlations of multiplicative functions}.
We now move on to correlations.
For $P,Q\in\mathbb{Z}[x],$ we define the local correlation
\begin{equation}\label{localmultiple}
M_p(f(P),g(Q))=\lim_{x\to\infty}\frac{1}{x}\sum_{n\le x}f_{p}(P(n))g_{p}(Q(n)) .
\end{equation}
Evaluating these local factors is also easy yet can be technically complicated, as we shall see below in the case that $P$ and $Q$ are both linear.

More generally we establish the following 
\begin{theorem}\label{intromain}
Let $f,g:\mathbb{N}\to\mathbb{U}$ be multiplicative functions. Let $P,Q\in\mathbb{Z}[x]$ be two polynomials, such that $\text{res}(P,Q)\ne 0.$ Then,
$$
\frac{1}{x}\sum_{n\le x}f(P(n)){g(Q(n))}=\prod_{p\le x}M_p(f(P),g(Q)) + \textrm{\rm Error}(f(P),g(Q),x)
$$
where
$$
\textrm{\rm Error}(f(P),g(Q),x) \ll \mathbb{D}_P(1,f;\log x;x)+\mathbb{D}_Q(1,g;\log x;x)+\frac{1}{\log \log x}\cdot
$$
 \end{theorem}
 
Theorem \ref{intromain} implies that if $\mathbb{D}(1,f;x),\mathbb{D}(1,g;x)<\infty$ and $\sum_{p\in N_P(x)}1-\operatorname{Re}(f(p^k))=o(x),$ $\sum_{p\in N_Q(x)}1-\operatorname{Re}(g(p^k))=o(x)$ then 
$$
\frac{1}{x}\sum_{n\le x}f(P(n)){g(Q(n))}=\prod_{p\le x}M_p(f(P),g(Q)) + o(1)=\prod_{p\ge 1}M_p(f(P),g(Q)) + o(1).
$$
If $\mathbb{D}_{P}(f,n^{it};\infty), \mathbb{D}_{P}(g,n^{iu};\infty)<\infty$ then we can write $f_0(n)=f(n)/n^{it}$ and $g_0(n)=g(n)/n^{iu}$ so that $\mathbb{D}_{P}(1,f_0;\infty), \mathbb{D}_{P}(1,g_0;\infty)<\infty$, We apply Theorem \ref{intromain} to the mean value of $f_0(P(n))g_0(Q(n))$, and then  proceed by partial summation to obtain
$$
\frac{1}{x}\sum_{n\le x}f(P(n)){g(Q(n))}=M_i(f(P),g(Q),x)  \prod_{p\le x}M_p(f_0(P),g_0(Q)) + \textrm{\rm Error}(f_0(P),g_0(Q),x)
$$
where, if $P(x)=ax^D+\ldots$ and $Q(x)=bx^d+\ldots$ then we define $T=Dt+du$ and
$$
M_i(f(P),g(Q),x) := \frac{1}{x}\sum_{n\le x} P(n)^{it}Q(n)^{iu}  \sim a^{it}b^{iu} \frac{x^{iT}}{1+iT}.
$$

The same method works for $m$-point correlations 
 \[
\sum_{n\le x}f_1(P_1(n))f_2(P_2(n))\cdot \dots \cdot f_m(P_m(n))\]
for multiplicative functions $f_j:\mathbb{N}\to\mathbb{U}$ and polynomials $P_j$ with each $\mathbb{D}_{P_j}(n^{it_j},f_j,\infty)<\infty.$  We give a more explicit version of our results in the case that  $P$ and $Q$ are linear polynomials:

\begin{corollary}\label{introlinear}
Let $f,g:\mathbb{N}\to\mathbb{U}$ be multiplicative functions for which $\mathbb{D(}f,n^{it},\infty),$ $\mathbb{D}(g,n^{iu},\infty)<\infty$, and write $f_0(n)=f(n)/n^{it}$ and $g_0(n)=g(n)/n^{iu}$.
Let  $a,b\ge 1$,\ $c,d$ be integers with $(a,c)=(b,d)=1$ and $ad\ne bc$. As above we have
$$
\frac{1}{x}\sum_{n\le x}f(an+c){g(bn+d))}=M_i(f(P),g(Q),x)\prod_{p\le x}M_p(f_0(P),g_0(Q)) +o(1).
$$
We have
\[M_i(f(P),g(Q),x)\sim\frac{a^{it}b^{iu}x^{i(t+u)}}{1+i(t+u)}\cdot\]
If $p|(a,b)$ then $M_p(f_0(P),g_0(Q))=1.$ If $p\nmid ab(ad-bc),$ then 
\[
M_p(f_0(P),g_0(Q))=M_p(f_0(P))+M_p(g_0(Q))-1=1+\left(1-\frac{1}{p}\right)\left(\sum_{j\ge 1}\frac{f_0(p^j)}{p^j}+\sum_{j\ge 1}\frac{g_0(p^j)}{p^j}\right)\cdot
\]
In general, if $p\nmid(a,b)$ we have a more complicated formula
\begin{align*}
M_p(f_0(P),g_0(Q))=\sum_{\substack{0\le i\le k,\\k\ge 0, \\p^k\vert| ad-bc}}\left(\frac{\theta(p^i){\gamma(p^i)}}{p^i}+\delta_b\sum_{j> i}\frac{\theta(p^i){\gamma(p^j)}}{p^{j}}+\delta_a\sum_{j> i}\frac{{\gamma(p^i)}\theta(p^j)}{p^{j}}\right)
\end{align*}
and  $\delta_{\l}=0$ when $p\vert \l$ and $\delta_{\l}=1$ otherwise. Here $f_0=1*\theta$ and  $g_0=1*\gamma.$
\end{corollary}

For $t=u=0,$ some version of Corollary~\ref{introlinear} also appeared in 
Hildebrand~\cite{MR965752}, Elliot~\cite{MR1292619}, Stepanauskas~\cite{MR1964870}.

Next we apply Theorem~\ref{intromain} to obtain a number of consequences.
Roughly speaking, the key idea for our applications is that by expanding 
\begin{align*}
\frac{1}{x}\sum_{n\le x}\left(\sum_{k=n+1}^{n+H+1}f(k)\right)^2&= \sum_{|h|\le H}(H-|h|)\sum_{n\le x}f(n)\overline{f(n+h)}+O\left(\frac{H^2}{x}\right)
\end{align*}
and then $h=0$ term equals to $H$ if each $|f(n)|=1.$ Therefore if the above sum is small then
\[\frac{1}{x}\sum_{n\le x}f(n)\overline{f(n+h)}\gg 1\]
for some $h,$ $1\le |h|\le H.$ As Tao showed, if some weighted version of this is true, then $\mathbb{D}(f(n),\chi(n)n^{it};x)\ll 1$ for some primitive character $\chi.$ Therefore, to understand the above better, we need to give a version of Theorem~\ref{intromain} for functions $f$ with $\mathbb{D}(f(n),\chi(n)n^{it};x)\ll 1.$ 
\bigskip
 
\noindent \textbf{1.3.\ Correlations with characters}.  Now we will suppose that  
$\mathbb{D}(f(n),n^{it}\chi(n),\infty)<\infty$ for some $t\in\mathbb{R}$ where $\chi$ is a primitive character of conductor $q.$ We define $F$ to be the multiplicative function such that
\begin{align*}
 F(p^k) = \begin{cases} f(p^k)\overline{\chi(p^k)}p^{-ikt}, & \mbox{if } p\nmid q \\ 1, & \mbox{if } p\mid q,\end{cases}
 \end{align*}
and
 $$M_p(F,\overline{F};d)=\lim_{x\to\infty}\frac{1}{x}\sum_{n\le x}F_p(n)\overline{F_p(n+d)}.$$
 In Section $3$ we prove 
\begin{theorem}\label{charactercor}
Let $f:\mathbb{N}\to\mathbb{U}$ be a multiplicative function such that $\mathbb{D}(f(n),n^{it}\chi(n);\infty)<\infty$ for some $t\in\mathbb{R}$ and $\chi$ is a primitive character of conductor $q.$ Then for any non-zero integer $d$ we have 

\[\frac{1}{x}\sum_{n\le x}f(n)\overline{f(n+d)}=\prod_{\substack{p\le x\\ p\nmid q}}M_p(F,\overline{F};d)\prod_{p^{\l}\vert| q}M_{p^{\l}}(f,\overline{f},d)+o(1),\]
where
\[M_{p^{\l}}(f,\overline{f},d) = \begin{cases} 0, & \mbox{if } p^{\l-1}\nmid d \\ 1-\frac{1}{p}, & \mbox{if } p^{\l-1}\vert| d \\ \left(1 -\frac 1p \right)  \sum_{j=0}^{k} \frac{|f(p^j)|^2}{p^j} -  \frac{|f(p^k)|^2}{p^{k}}, & \mbox{if } p^{\l+k}\vert| d\end{cases}\]
for any $k\ge 0$ and if $p^n\vert| d,$ then
\[M_p(F,\overline{F},d)=1-\frac{2}{p^{n+1}}+\left(1-\frac{1}{p}\right)\sum_{j>n}\left(\frac{F(p^n)\overline{F(p^j)}}{p^j}+\frac{\overline{F(p^n)}F(p^j)}{p^j}\right).\]
In particular, the mean value is $o(1)$ if $q\nmid d\prod_{p\vert q}p.$
\end{theorem}


The same method works for correlations 
$$\sum_{n\le x}f(n)g(n+m)$$
where  $\mathbb{D}(f(n),n^{it}\chi(n);\infty),\ \mathbb{D}(g(n),n^{iu}\psi(n);\infty)<\infty.$ 

\noindent \textbf{1.4.\ The Erd\H{o}s discrepancy problem for multiplicative functions}.  The Polymath5 project showed, using Fourier analysis, that  the  Erd\H{o}s discrepancy problem can be reduced to a  statement about completely multiplicative functions. In particular, Tao~\cite{MR3} established that for any completely multiplicative 
$f:\mathbb{N}\to \{-1,1\}$,
\[\limsup_{x\to\infty}\sum_{n\le x}f(n)=\infty.\]
In \cite{MR0098702}, \cite{MR797781}, \cite{MR851041}, Erd\H{o}s  along with the Erd\H{o}s discrepancy problem, asked to classify all multiplicative $f:\mathbb{N}\to \{-1,1\}$ such that
\begin{equation}\label{boundedsum}\limsup_{x\to\infty}\left|\sum_{n\le x}f(n)\right|<\infty.\end{equation}
In~\cite{MR3}, Tao, partially answering this question, proved that if for a multiplicative $f:\mathbb{N}\to \{-1,1\},$~\eqref{boundedsum} holds, then 
$f(2^j)=-1$ for all $j,$ and 
\begin{equation}\label{taopartial}
\sum_{p}\frac{1-f(p)}{p}<\infty.
\end{equation}
In Section $4,$ we resolve this question completely by proving
 \begin{theorem}\label{tao}{\bf [Erd\H{o}s-Coons-Tao conjecture]}
Let $f:\mathbb{N}\to\{-1,1\}$ be a multiplicative function. Then~\eqref{boundedsum} holds if and only if 
there exists an integer $m\ge 1$ such that $f(n+m)=f(n)$ for all $n\ge 1$ and $\sum_{n=1}^m f(n)=0.$
\end{theorem}

There are examples known  with bounded sums, such as the multiplicative function $f$ for which 
$f(n)=+1$ when $n$ is odd and $f(n)=-1$ when $n$ is even.  One can easily show $f$ satisfies the above hypotheses if and only if $m$ is even, $f(2^k)=-1$ for all $k\geq 1$, and  $f(p^k)=f((p^k,m))$ for all odd prime powers $p^k.$ In particular if $p$ does not divide $m$ then $f(p^k)=1.$ 

It would be interesting to classify all complex valued multiplicative $f:\mathbb{N}\to\mathbb{T}$ for which~\eqref{boundedsum} holds. Using Theorem~\ref{charactercor} it easy to prove
\begin{theorem}\label{complexcor}
Suppose for a multiplicative $f:\mathbb{N}\to\mathbb{T},$~\eqref{boundedsum1} holds. Then there exists a primitive character $\chi$ of an odd conductor $q$ and $t\in\mathbb{R},$ such that $\mathbb{D}(f(n),\chi(n)n^{it};\infty)<\infty$ and $f(2^k)=-\chi^k(2)2^{-ikt}$ for all $k\ge 1.$
\end{theorem}
 \bigskip

\noindent \textbf{1.5.\ Distribution of $(f(n),f(n+1))$}.  
Let $f:\mathbb{N}\to\mathbb{C}$ be a multiplicative function
and $\triangle f(n)=f(n+1)-f(n).$ K\'atai conjectured and Wirsing proved (first in a letter to K\'atai, and then in a joint paper with Tan and Shao~\cite{MR1373561}) that if a unimodular multiplicative function $f$ satisfies 
$\triangle f(n)\to 0$ then $f(n)=n^{it}$ (see also a nice paper of Wirsing and Zagier~\cite{MR1864627} for a simpler proof). One would naturally expect that if $\triangle f(n)\to 0$ in some averaged sense, than the similar conclusion must hold. K\'atai~\cite{MR729295} made the following conjecture which we prove in Section $5:$
\begin{theorem}\label{introkatai}{\bf [K\'atai's Conjecture, 1983]}
If $f:\mathbb{N}\to\mathbb{C}$ is a multiplicative function and
\begin{align*}
\lim_{x\to\infty}\frac{1}{x}\sum_{n\le x}|\triangle f(n)|=0
\end{align*}
then either 
\[\lim_{x\to\infty}\frac{1}{x}\sum_{n\le x}|f(n)|=0\] 
or
$f(n)=n^{s}$ for some $\operatorname{Re}(s)<1.$
\end{theorem}
Since $f(n)=\text{e}^{h(n)}$ is multiplicative, where $h(n):\mathbb{N}\to\mathbb{R}$ is an additive function, one may compare Theorem~\ref{introkatai} with the following statement about additive functions, first conjectured by Erd\H{o}s~\cite{MR0015424}  and proved later by K\'atai~\cite{MR0250991} (and independently by Wirsing):
if $h:\mathbb{N}\to\mathbb{C}$ is an additive function and
\[\lim_{x\to \infty}\frac{1}{x}\sum_{n\le x}|h(n+1)-h(n)|=0,\]
then $h(n)=c\log n.$

The conjecture attracted considerable attention of several authors including K\'atai, Hildebrand, Phong and others. See, for example~\cite{MR932664}, \cite{MR3158863},\cite {MR1798718},\cite{MR1153489} for some of the results and the survey paper~\cite{MR1764803} with an extensive list of the related references.

 \bigskip

\noindent \textbf{1.6.\ Binary additive problems}.   A sequence $A$ of positive integers is called multiplicative, if its characteristic function, $1_A$, is multiplicative. We define
\[\rho_A(d)=\lim_{x\to\infty}\frac{1}{x/d}\sum_{k\le x/d}\text{I}_A(kd),\]
with $\rho_A=\rho_A(1)$, which is the density of $A$. Note that these constants all exist by Wirsing's Theorem.

Binary additive problems, which involve estimating quantities like
\[
r(n)=|\{(a,b)\in A\times B:\ a+b=n\}|
\] 
are considered difficult. However, using a variant of circle method Br{\"u}dern~\cite{MR2508639}, among other things, established the following theorem, which we will deduce from Theorem~\ref{intromain} in section 6.

\begin{theorem}\label{brudern}{\bf{[Br{\"u}dern, 2008]}}
Suppose $A$ and $B$ are multiplicative sequences of positive density $\rho_A$ and $\rho_B$ respectively. For $k\ge 1,$ let 
\[a(p^k)=\rho_A(p^k)/p^k-\rho_A(p^{k-1})/p^{k-1}\]
Define $b(p^k)$ in the same fashion. Then, $$r(n)=\rho_A\rho_B\sigma(n)n+o(n)$$
where
\[\sigma(n)=\prod_{p^m\vert| n}\left(1+\sum_{k=1}^m\frac{p^{k-1}a(p^k)b(p^k)}{p-1}-\frac{p^{m}a(p^{m+1})b(p^{m+1})}{(p-1)^2}\right)\cdot\]
\end{theorem}
{\bf Acknowledgement.} I would like to thank Andrew Granville for all his support and encouragement as well as many valuable comments and suggestions. The research leading to the results of this paper received funding from the NSERC grant and the ISM doctoral award. 
\end{section}

\begin{section}{Multiplicative functions of polynomials}
 For any given polynomial $P(x)\in\mathbb{Z}[x]$ we define $\omega_P(p^k)$ to be the number of solutions of $P(x)=0(\text{mod} (p^k)).$ Clearly, $\omega_P(p^k)\le \deg P$ for all but finitely many primes $p.$
 We begin by showing that the mean value of $f(P(n))$ in general significantly depends on the large primes. We restrict ourselves to the case $P(x)=x^2+1$ but the same arguments work for all polynomials $P(x)\in\mathbb{Z}[x]$ that are not product of linear factors.
 
 \begin{lemma}\label{largeprimes}
Let $P(x)=x^2+1.$ For any $x\ge 2,$ and any complex numbers $g(p^k)\in\mathbb{T},$ $p\le 2x,$ $k\ge 1,$ there exists a multiplicative function $f:\mathbb{N}\to\mathbb{T}$ such that $f(p^k)=g(p^k)$ for all $p\le 2x$ and 
\[\left|\frac{1}{x}\sum_{n\le x}f(P(n))\right|\ge \frac{1}{2}+o(1).\]
\end{lemma}

\begin{proof}
Let $$\textfrak{M}(x) = \{ n_p\le x \mid \exists p\in N_P(x), p\vert P(n_p) \}.$$
We note that for each $p\ge 2x,$ there exists at most one element $n_p\in\textfrak{M}(x)$ such that $p\vert P(n_p)$ and moreover all prime factors of $P(n_p)/p$ are smaller than $x.$
We have
\begin{align*}
2x\log x+O(x)&=\sum_{n\le x}\log P(n)=\sum_{n\le x}\sum_{d\vert P(n)}\Lambda(d)\\&\le 2\sum_{\substack{p\le x,\\ p=1\ \text{mod} (4)}}\log p\cdot\frac{x}{p}+\sum_{\substack{p>2x,\\ p|P(n_p),\\ n_p\le x}}\log p+O(x)\\&\le x\log x+2\log x\cdot |\textfrak{M}(x)|+O(x)
\end{align*}
and therefore
$$|\textfrak{M}(x)|\ge x\left(\frac{1}{2}+o(1)\right).$$

Consider the multiplicative function $f$ defined as follows: $f(p^k)=g(p^k)$ for all primes $p\le 2x$ and  
$$f(p)=e^{i\phi}\overline{f\left(\frac{P(n_p)}{p}\right)}$$ if $p> 2x$ and there exists $n_p\in \textfrak{M}(x)$ such that $p\vert P(n_p),$
where $$\phi=\text{arg}\left(\sum_{\substack{n\in\overline{\textfrak{M}(x)}\\ n\le x}}f(P(n))\right).$$ Define $f(p^k)=1$ for all other primes and all $k\ge 1.$
 Clearly,
\begin{align*}
\sum_{n\le x}f(P(n))=\sum_{\substack{n\in\overline{\textfrak{M}(x)}\\ n\le x}}f(P(n))+\sum_{n_p\in \textfrak{M}(x)}f(P(n_p))=\sum_{\substack{n\in\overline{\textfrak{M}(x)},\\ n\le x}}f(P(n))+e^{i\phi}|\textfrak{M}(x)|.\end{align*}
Selecting $\phi$ so that the two sums point in the same direction, we deduce that
\[\left|\frac{1}{x}\sum_{n\le x}f(P(n))\right|\ge \frac{|\textfrak{M}(x)|}{x}\ge \frac{1}{2}+o(1).\]
\end{proof}

\begin{propn}[\ref{dependence}]
There exists a multiplicative function $f:\mathbb{N}\to [-1,1]$ such that $\mathbb{D}(1,f;x)=2\log\log x+O(1)$ for all $x\ge 2$ and 
\[\limsup_{x\to\infty}\left|\frac{1}{x}\sum_{n\le x}f(n^2+1)\right|\ge \frac{1}{2}+o(1).\]
\end{propn}

\begin{proof}
Take the sequence $x_k=2^{{2^k}}$ for $k\ge 1$ and define completely multiplicative function $f$ inductively: $f(p)=-1$ for all  primes in $p\in (x_k,x_{k+1}]$ unless $p\in N_P(x_k)$, in which case we define the function  as in the proof of Lemma~\ref{largeprimes}. This guarantees that  for all $k\ge 1,$
 \[\left|\frac{1}{x_k}\sum_{n\le x_k}f(n^2+1)\right|\ge \frac{1}{2}+o(1).\]
 Since $N_P(x)$ contains at most $x$ elements, we have $\sum_{p\in N_P(x)} 1/p \leq \sum_{x<p\leq 2x\log x} 1/p \ll (\log\log x)/\log x$, so that 
 $\sum_{k\geq 1} \sum_{p\in N_P(x_k)} 1/p \ll \sum_{k\geq 1} k/2^k\ll 1$. Therefore
 $$
 \mathbb{D}(1,f;x)\ge \sum_{\substack{p\le x\\ p\notin \cup_{k\ge 1}N_P(x_k) }}\frac{2}{p}\ge 2\log\log x-O(1).
 $$
\end{proof}

We thus focus on the class of functions such that $f(p)$ is close to $1$ on large primes $p\ge x$ where the distance is given by  $D_P(1,f;x)$ where
\[
D_P(1,f;x)^2 \asymp \sum_{p}(1-\operatorname{Re}f(p^k))\cdot\frac{1}{x}\sum_{\substack{n\le x,\\ p^k\vert| P(n)}}1 ,
\]
which generalizes $D(1,f;x)$ where
\[
D(1,f;x)^2 \asymp D^*(1,f;x)^2 \asymp \sum_{p}(1-\operatorname{Re}f(p^k))\cdot\frac{1}{x}\sum_{\substack{n\le x,\\ p^k\vert| n}}1 
\]

 In order to prove Theorem~\ref{intromain}, we begin by proving a few auxiliary results. The following lemma is a simple consequence of the Erd\H{o}s-Kac type theorem for the polynomial sequences.
 
\begin{lemma}\label{erdoskacpol}
Let $h:\mathbb{N}\to\mathbb{C}$ be an additive function such that $h_s(p^k)=0$ for $p^k\ge x$ and $|h(p^k)|\le 2$ for all $p$ and $k\ge 1.$ Suppose $P(x)\in\mathbb{Z}[x]$ is irreducible. Define 
\[\mu_{h,P}=\sum_{p^k\le x}\frac{h(p^k)}{p^k}\left(\omega_P(p^k)-\frac{\omega_P(p^{k+1})}{p}\right)\]
and 
\[\sigma_{h,P}^2=\sum_{p^k\le x}\frac{|h(p^k)|^2}{p^k}\left(\omega_P(p^k)-\frac{\omega_P(p^{k+1})}{p}\right).\]
Then 
\begin{equation}\label{kacpol}
\sum_{n\le x}|h(P(n))-\mu_{h,P}|^2\ll x\sum_{p^k\le x}\frac{|h(p^k)|^2}{p^k}+x\frac{(\log\log x)^3}{\log x}\cdot
\end{equation}
\end{lemma}
\begin{proof}
By multiplicativity, we have 
\[|\{ n\le x \mid d\vert P(n) \}|=\frac{\omega_P(d)}{d}x+r_d\]
where $r_d=O(\omega_P(d)).$
Furthermore, by Proposition $4$ from~\cite{MR2290492} applied to the additive functions in place of strongly additive
\begin{align*}
\sum_{n\le x}|h(P(n))-\mu_{h,P}|^2&\le C_2x\sigma_{h,P}^2+O\left(\max_{p\le y}|h(p^k)|^2\left(\sum_{p\le x}\frac{\omega_P(p)}{p}\right)^2\sum_{\substack{d=p_1p_2,\\ p_i\le x }}|r_d|\right).
\end{align*}
 The error term is bounded by
\[\max_{p\le x}|h(p^k)|^2\left(\sum_{p\le x}\frac{\omega_P(p)}{p}\right)^2\sum_{\substack{d=p_1p_2,\\ p_i\le x }}|r_d|\ll \max_{p\le x}|h(p^k)|^2 ( \log\log x)^2\cdot\frac{x\cdot\log\log x}{\log x}\cdot\]
Combining this observation with the estimate
\[\sigma_{h,P}^2\ll \sum_{p^k\le x}\frac{|h(p^k)|^2}{p^k} \]
we conclude the proof of~\eqref{kacpol}.
\end{proof}
In what follows, we are going to focus on two-point correlations but the same method actually works for $m-$ point correlations  with mostly notational modifications.
Let
\[\mu_{h,P}=\sum_{p^k\le x}h(p^k)\left(\frac{\omega_P(p^k)}{p^k}-\frac{\omega_P(p^{k+1})}{p^{k+1}}\right)\]
and
\[\textfrak{P}(f;P;x)=\prod_{p\le x}\left(\sum_{k\ge 0}f(p^k)\left(\frac{\omega_P(p^k)}{p^k}-\frac{\omega_P(p^{k+1})}{p^{k+1}}\right)\right).\]
We also introduce equivalent distance 
 \[
 \mathbb{D}^*_P(f,g;y;x)=\left(\sum_{y\le p^k\le x}\frac{1-\operatorname{Re}{(f(p^k)\overline{g(p^k)}})}{p^k}+\sum_{p\in N_P(x)}\frac{1-\operatorname{Re}{(f(p^k)\overline{g(p^k)}})}{x}\right)^{\frac{1}{2}}.
 \]
The following proposition plays crucial role.
\begin{prop}\label{key} Let $f(n)$ be a multiplicative function and $g(n)$ be any sequence such that  $|f(n)|\le 1$ and $|g(n)|\le 1$ for all $n\in\mathbb{N}.$ Let $P(n)\in\mathbb{Z}[x].$ Then
\[\sum_{ n\le x}f(P(n))g(n)=\textfrak{P}(f;P;x)\sum_{n\le x}g(n)+O\left(x\mathbb{D}^*_P(1,f;x)+\frac{x(\log\log x)^{\frac{3}{2}}}{\sqrt{\log x}}\right).\]
\end{prop}
\begin{proof}
We begin by proving the proposition for the multiplicative function $f$ such that $f(p^k)=1$ for all $p^k\ge x.$ Note $e^{z-1}=z+O(|z-1|^2)$ for $|z|\le 1.$ By repeatedly applying triangle inequality we have that for all $|z_i|,|w_i|\le 1$ 
\begin{equation}\label{iterate1}
\left|\prod_{1\le i\le n}z_i-\prod_{1\le i\le n}w_i\right|\le \sum_{1\le i\le n}|z_i-w_i|.\end{equation}
Therefore,
\begin{align*}
\prod_{p^k\vert| P(n)}e^{f(p^k)-1}=\prod_{p^k\vert| P(n)}\left(f(p^k)+O(|f(p^k)-1|^2)\right)=\prod_{p^k\vert| P(n)}f(p^k)+O\left(\sum_{p^k\vert| P(n)}|f(p^k)-1|^2\right)
\end{align*}
and
\[f(P(n))=\prod_{p^k\vert| P(n)}f(p^k)=\prod_{p^k\vert| P(n)}e^{f(p^k)-1}+O\left(\sum_{p^k\vert| P(n)}|f(p^k)-1|^2\right).\]
We now introduce an additive function $h,$ such that $h(p^k)=f(p^k)-1.$ Clearly,
\begin{align*}
\sum_{n\le x}f(P(n))g(n)-\sum_{n\le x}&g(n)e^{h(P(n))}\ll \sum_{n\le x}\sum_{\substack{p^k\vert| P(n),\\ p^k\le x}}|f(p^k)-1|^2&\\&\ll x\sum_{p^k\le x}\frac{|f(p^k)-1|^2}{p^k}\ll x\mathbb{D}^*(f,1;x)^2.
\end{align*}
Since 
$|e^a-e^b|\ll|a-b|$ for $\operatorname{Re}{(a)},\operatorname{Re}{(b)}\le 0,$
Cauchy-Schwarz together with Lemma~\ref{erdoskacpol} imply 
\begin{align*}
\sum_{n\le x}g(n)e^{h(P(n))}-e^{\mu_{h,P}}\sum_{n\le x}g(n)&\ll \sum_{n\le x}|e^{h(P(n))}-e^{\mu_{h,P}}|\ll \sum_{n\le x}|h(P(n))-\mu_{h,P}|\\&\le (x\sum_{n\le x}|h(P(n))-\mu_{h,P}|^2)^{1/2}\ll x\mathbb{D}^*(f,1;x)+\frac{x(\log\log x)^{\frac{3}{2}}}{\sqrt{\log x}}\cdot
\end{align*}
We introduce $\mu_{h,P}=\sum_{p\le x}\mu_{h,p},$ where 
\[\mu_{h,p}=\sum_{p^k\le x}h(p^k)\left(\frac{\omega_P(p^k)}{p^k}-\frac{\omega_P(p^{k+1})}{p^{k+1}}\right)\]
 and observe
 \[e^{\mu_{h,p}}=1+\mu_{h,p}+O(\mu_{h,p}^2)=\sum_{1\le p^k\le x}f(p^k)\left(\frac{\omega_P(p^k)}{p^k}-\frac{\omega_P(p^{k+1})}{p^{k+1}}\right)+O\left(\frac{1}{x}+\frac{1}{p}\sum_{p^k\le x}\frac{|h(p^k)|}{p^k}\right)\cdot\]
 Note that $|e^{\mu_{h,p}}|\le 1.$
 Using~\eqref{iterate1} and the Cauchy-Schwarz inequality once again yields 
 \begin{align*}
 |e^{\mu_{h,P}}-\textfrak{P}(f;P;x)|&\le\sum_{p\le x}\left|e^{\mu_{h,p}}-\sum_{1\le p^k\le x}f(p^k)\left(\frac{\omega_P(p^k)}{p^k}-\frac{\omega_P(p^{k+1})}{p^{k+1}}\right)+O\left(\frac{1}{x}\right)\right|
 \\&\ll \sum_{p^k\le x}\frac{1}{p}\frac{|f(p^k)-1|}{p^k}+\sum_{p\le x}\frac{1}{x}\ll \mathbb{D}^*(f,1;x)+\frac{1}{\log x}
 \end{align*}
 which completes the proof of the lemma in the special case when $f(p^k)=1$ for $p^k\ge x.$ 
 
 We now consider any multiplicative function $f$ and decompose $f(n)=f_{s}(n)f_{\l}(n)$ where 
\[f_{s}(p^k) = \begin{cases} f(p^k), & \mbox{if } p^k\le x \\ 1, & \mbox{if } p^k>x\end{cases}\]
and
\[f_{\l}(p^k) = \begin{cases} 1, & \mbox{if } p^k\le x \\ f(p^k), & \mbox{if } p^k>x.\end{cases}.\]
Note that for a fixed prime power $p^k\in N_P(x),$
$$|\{ n\le x \mid p^k\vert P(n) \}|\le \omega_P(p^k)$$
and each $P(n)$ is divisible by $\ll \deg P$ elements of $N_P(x).$
Using Cauchy-Schwarz
\begin{align*}
\sum_{n\le x}f(P(n))g(n)-\sum_{n\le x}&f_{s}(P(n))g(n)\ll \sum_{n\le x}\sum_{\substack{p^k\vert| P(n),\\ p^k\ge x}}|f(p^k)-1|&\ll x\left(\sum_{p^k\in N_P(x)}\frac{|f(p^k)-1|^2}{x} \right)^{\frac{1}{2}}.
\end{align*}
We are left to collect the error terms and note that 
\[\mathbb{D}^*(1,f;x)+\left(\sum_{p^k\in N_P(x)}\frac{1-\operatorname{Re}(f(p^k))}{x}\right)^{\frac{1}{2}}\le 2\mathbb{D}_P^*(1,f;x).\]
 \end{proof}
Let $f,g:\mathbb{N}\to\mathbb{U}$ be multiplicative functions. For any two irreducible polynomials $P,Q\in\mathbb{Z}[x]$ we define
\[M(f,g;x)=\frac{1}{x}\sum_{n\le x}f(P(n)){g(Q(n))}.\]
We define $\omega(p^k,p^{\l})$ to be the quantity such that
$$\{ n\le x \mid  p^k\vert| P(n), p^{\l}\vert| Q(n)\}=x\omega(p^k,p^{\l})+O(1).$$
We note that if $p\nmid\text{res}(P,Q)$ then $\omega(p^k,p^{\l})=0$ unless $k=0$ or $\l=0.$ In the latter case,
\[\omega(p^k,1)=\frac{\omega_P(p^k)}{p^k}-\frac{\omega_P(p^{k+1})}{p^{k+1}}\]
and
\[\omega(1,p^{\l})=\frac{\omega_Q(p^{\l})}{p^{\l}}-\frac{\omega_Q(p^{\l+1})}{p^{\l+1}}\cdot\]
Furthermore, by the Chinese Remainder Theorem we have 
\begin{align*}
\{ n\le x \mid  d_1\vert P(n), d_2\vert Q(n)\}&=xF(d_1,d_2)+O(\omega_P(d_1)\omega_{Q}(d_2))=xF(d_1,d_2)+O(x^{\varepsilon}).\end{align*}
for some multiplicative function $F(d_1,d_2).$ 
Our main goal in this section is to prove that the mean value $M(f,g;x)$ satisfies the ``local-to-global" principle.
We first evaluate the local correlations.
\begin{lemma}\label{localcorrelations1}
Let $f,g:\mathbb{N}\to\mathbb{U}$ be multiplicative functions. Define $f_p,g_p$ as in~\eqref{localfactor}. Let $P,Q\in\mathbb{Z}[x]$ and $\text{res}(P,Q)\ne 0.$  Then,
\begin{align*}
\frac{1}{x}\sum_{n\le x}f_p(P(n))g_p(Q(n))=\sum_{\substack{p^k,p^{\l}\ge 1
}}f(p^k)g(p^{\l})\omega(p^k,p^{\l})+O\left(\frac{\log x}{x\log p}\right).
\end{align*}
In particular, if $p\nmid \text{res}(P,Q),$ then 
\begin{align*}
\frac{1}{x}&\sum_{n\le x}f_p(P(n))g_p(Q(n))\\&=\left(\sum_{k\ge 0}f(p^k)\left(\frac{\omega_P(p^k)}{p^k}-\frac{\omega_P(p^{k+1})}{p^{k+1}}\right)+\sum_{k\ge 0}g(p^k)\left(\frac{\omega_Q(p^k)}{p^k}-\frac{\omega_Q(p^{k+1})}{p^{k+1}}\right)-1\right)+O\left(\frac{\log x}{x\log p}\right).
\end{align*}
\end{lemma}
\begin{proof}
We first suppose that $p\nmid\text{res}(P,Q).$ In this case we have 
\begin{align*}
\frac{1}{x}&\sum_{n\le x}f_p(P(n))g_p(Q(n))=\frac{1}{x}\left(\sum_{\substack{p^k\le x,\\ p^k\vert| P(n)}}f(p^k)+\sum_{\substack{p^{\l}\le x,\\ p^{\l}\vert| Q(n)}}g(p^{\l})+\sum_{\substack{n\le x,\\ p^{0}\vert|P(n)Q(n)}}1\right)\\&=\left(\sum_{k\ge 0}f(p^k)\left(\frac{\omega_P(p^k)}{p^k}-\frac{\omega_P(p^{k+1})}{p^{k+1}}\right)+\sum_{k\ge 0}g(p^k)\left(\frac{\omega_Q(p^k)}{p^k}-\frac{\omega_Q(p^{k+1})}{p^{k+1}}\right)-1\right)+O\left(\frac{\log x}{x\log p}\right).
\end{align*}
More generally,
\begin{align*}
\frac{1}{x}\sum_{n\le x}f_p(P(n))g_p(Q(n))=\frac{1}{x}\sum_{\substack{p^k,p^{\l}\le x,\\ p^k\vert| P(n),\\   p^{\l}\vert| Q(n)
}}f(p^k)g(p^{\l})=\sum_{\substack{p^k,p^{\l}\ge 1
}}f(p^k)g(p^{\l}){\omega(p^k,p^{\l})}+O\left(\frac{\log x}{x\log p}\right).
\end{align*}
This completes the proof of the lemma.
\end{proof}
 \begin{theoremn}[\ref{intromain}]
Let $f,g:\mathbb{N}\to\mathbb{U}$ be multiplicative functions. Let $P,Q\in\mathbb{Z}[x]$ be two polynomials, such that $\text{res}(P,Q)\ne 0.$ Then,
$$
\frac{1}{x}\sum_{n\le x}f(P(n)){g(Q(n))}=\prod_{p\le x}M_p(f(P),g(Q)) + \textrm{\rm Error}(f(P),g(Q),x)
$$
where
$$
\textrm{\rm Error}(f(P),g(Q),x) \ll \mathbb{D}_P(1,f;\log x;x)+\mathbb{D}_Q(1,g;\log x;x)+\frac{1}{\log \log x}\cdot
$$
\end{theoremn}
\begin{proof}
Choose $y= (1-\varepsilon)\log x.$ We begin by decomposing $f(n)=f_{s}(n)f_{\l}(n)$ where 
\[f_{s}(p^k) = \begin{cases} f(p^k), & \mbox{if } p^k\le y \\ 1, & \mbox{if } p^k>y\end{cases}\]
and
\[f_{\l}(p^k) = \begin{cases} 1, & \mbox{if } p^k\le y \\ f(p^k), & \mbox{if } p^k>y.\end{cases}\]
By analogy, we write $g(n)=g_s(n)g_{\l}(n).$
We apply Proposition~\ref{key} to get
\begin{align*}
\sum_{n\ge 1}f_{\l}(P(n))f_s(P(n))g(Q(n))=\textfrak{P}&(f_{\l};P;x)\sum_{n\le x}{f_s(P(n))g(Q(n))}\\&+O\left(x\mathbb{D}_{P}^*(1,f_{\l};y;x)+\frac{x(\log \log x)^{\frac{3}{2}}}{\sqrt{\log x}}\right).
\end{align*}
We now apply Proposition~\ref{key} to the inner sum to arrive at
\begin{align*}
\sum_{n\le x}g_{\l}(Q(n))g_s(Q(n))f_s(P(n))&=\textfrak{P}({g_{\l}};Q;x)\sum_{n\le x}f_s(P(n))g_s(Q(n))\\&+O\left(x\mathbb{D}_{P}^*(1,f_{\l};y;x)+x\mathbb{D}_{Q}^*(1,g_{\l};y;x)+\frac{x(\log \log x)^{\frac{3}{2}}}{\sqrt{\log x}}\right).
\end{align*}
Combining the last two identities we conclude
\begin{align*}
\sum_{n\le x}f(P(n)){g(Q(n))}&=\textfrak{P}(f_{\l};P;x)\textfrak{P}(g_{\l};Q;x)\sum_{n\le x}f_{s}(P(n)){g_{s}(Q(n))}\\&+O\left(x\mathbb{D}_{P}^*(1,f_{\l};y;x)+x\mathbb{D}_{Q}^*(1,g_{\l};y;x)+\frac{x(\log \log x)^{\frac{3}{2}}}{\sqrt{\log x}}\right).
\end{align*}
Let $f_{s}=1*\theta_{s},$ $g_s=1*\gamma_{s}.$ Then $\theta_{s}(p^k)=0$ and $\gamma_s(p^k)=0$ whenever $p^k\ge y.$
Since 
$\prod_{p^k\le y}p=e^{y+o(y)}\le x$
as long as $y\le (1-\varepsilon)\log x$ the following sums are supported on the integers $d_1,d_2\le x.$ Hence, 
\begin{align*}
\sum_{n\le x}f_{s}(P(n))&{g_{s}(Q(n)}=\sum_{\substack{d_1,d_2\le x,\\ p|d_i\Rightarrow p\le y}}\theta_{s}(d_1){\gamma_{s}(d_2)}\sum_{\substack{n\le x,\\ d_1|P(n),\\ d_2|Q(n)}}1\\&=\sum_{\substack{d\le x,\\d|\text{res}(P,Q)}}\sum_{\substack{d_1,d_2\le x,\\ (d_1,d_2)=d,\\ p|d_i\Rightarrow p\le y}}\theta_{s}(d_1){\gamma_{s}(d_2)}F(d_1,d_2)x+O\left(x^{\varepsilon}\sum_{d_1,d_2\le x}|\theta_{s}(d_1){\gamma_{s}(d_2)}|\right)\\&=\sum_{\substack{d\le x,\\d|\text{res}(P,Q)}}\sum_{\substack{d_1,d_2\ge 1,\\ (d_1,d_2)=d,\\ p|d_i\Rightarrow p\le y}}\theta_{s}(d_1){\gamma_{s}(d_2)}F(d_1,d_2)x+O\left(x^{\varepsilon}\sum_{d_1,d_2\le x}|\theta_{s}(d_1){\gamma_{s}(d_2)}|\right).
\end{align*}
To estimate the error term we observe
\begin{align}\label{error}
\sum_{d_1,d_2\le x}|\theta_{s}(d_1){\gamma_{s}(d_2)}|&\le x^{\frac{1}{2}}\left(\sum_{d\ge 1}\frac{|\theta_{s}(d)|}{d^{\frac{1}{4}}}\right)\left(\sum_{d\ge 1}\frac{|\gamma_{s}(d)|}{d^{\frac{1}{4}}}\right)\\&\le x^{\frac{1}{2}}\left(\prod_{p\le y}\left(\sum_{k\ge 0}\frac{|\theta_{s}(p^k)|}{p^{\frac{k}{4}}}\right)\right)\left(\prod_{p\le y}\left(\sum_{k\ge 0}\frac{|\gamma_{s}(p^k)|}{p^{\frac{k}{4}}}\right)\right)\nonumber\\&\ll x^{\frac{1}{2}}\left(\prod_{p\le y}\left(1+\frac{2}{p^{\frac{1}{4}}}\right)\right)^2\ll x^{\frac{1}{2}}\exp\left(\frac{3y^{3/4}}{\log y}\right)\cdot\nonumber
\end{align}
The last sum is $O(x^{\frac{1}{2}+\varepsilon})$ for $y\ll \log x$ and $y\to\infty.$
It easy to see that for $p\le y,$ Lemma~\ref{localcorrelations1} implies
\[M_p(f,g)=\sum_{p^k,p^{\l}\ge 1}\theta(p^k)\gamma(p^{\l})F(p^k,p^{\l}),\]
where $M_p(f,g)$ defined as in~\eqref{localmultiple}.
By multiplicativity the contribution of small primes is
\begin{align}\label{small correlation}
\sum_{\substack{d|\text{res}(P,Q)}}\sum_{\substack{d_1,d_2\ge 1,\\ (d_1,d_2)=d,\\ p|d_i\Rightarrow p\le y}}\theta_{s}(d_1){\gamma_{s}(d_2)}F(d_1,d_2)=\prod_{p\le y}M_p(f,g).
\end{align}
We are left to estimate $\textfrak{P}({f_{\l}};P;x)\textfrak{P}({g_{\l}};Q;x).$ The contribution of primes $p^k>y$ and $p\le y$ is 
\begin{align*}
\prod_{\substack{p^k\ge y,\\ p<y}}\left(1+\sum_{i\ge k}\frac{\theta_{\l}(p^k)\omega_P(p^k)}{p^k}\right)\prod_{\substack{p^k\ge y,\\ p<y}}\left(1+\sum_{i\ge k}\frac{\gamma_{\l}(p^k)\omega_Q(p^k)}{p^k}\right)&=1+O\left(\sum_{\substack{p^k\ge y\\ p<y}}\frac{1}{p^k}\right)&\\=1+O\left(\frac{1}{y}\cdot \frac{y}{\log y}\right)=1+O\left(\frac{1}{\log y}\right).
\end{align*}
 Furthermore, for $p\ge y$ we clearly have $(p,\text{res}(P,Q))=1$ and   
\begin{align*}
&\textfrak{P}({f_{\l}};P;x)\textfrak{P}({g_{\l}};Q;x)\\&= \left(1+O\left(\frac{1}{\log y}\right)\right)\cdot \prod_{y<p\le x}\left(1+\sum_{k\ge 1}\frac{\theta_{\l}(p^k)\omega_P(p^k)}{p^k}\right)\prod_{y<p\le x}\left(1+\sum_{k\ge1}\frac{\gamma_{\l}(p^k)\omega_Q(p^k)}{p^k}\right)\\&=\left(1+O\left(\frac{1}{\log y}\right)\right)\\&\times
\prod_{y<p\le x}\left(1+\sum_{k\ge 1}\frac{\theta(p^k)\omega_P(p^k)}{p^k}+\sum_{k\ge 1}\frac{{\gamma(p^k)\omega_Q(p^k)}}{p^k}+\sum_{k\ge 1}\frac{\theta(p^k)\omega_P(p^k)}{p^k}\sum_{k\ge 1}\frac{{\gamma(p^k)\omega_Q(p^k)}}{p^k}\right)\\&=\left(1+O\left(\frac{1}{\log y}\right)\right)
\exp\left(O\left(\sum_{y\le p\le x}\frac{1}{p^2}\right)\right)\prod_{y<p\le x}\left(1+\sum_{k\ge 1}\frac{\theta(p^k)\omega_P(p^k)}{p^k}+\sum_{k\ge 1}\frac{{\gamma(p^k)\omega_Q(p^k)}}{p^k}\right)\\&
=\left(1+O\left(\frac{1}{\log y}\right)\right)\prod_{y<p\le x}\left(1+\sum_{k\ge 1}\frac{\theta(p^k)\omega_P(p^k)}{p^k}+\sum_{k\ge 1}\frac{{\gamma(p^k)\omega_Q(p^k)}}{p^k}\right)
\end{align*}
and thus
$$\textfrak{P}({f_{\l}};P;x)\textfrak{P}({g_{\l}};Q;x)=\prod_{p\ge y}M_p(f,g)+O\left(\frac{1}{\log y}\right).$$ We note that $D^*_P(1,f;\log x;x)$ can be replaced with $D_P(1,f;\log x;x)$ at a cost $O(\frac{\log \log x}{\log x}).$
 Combining all of the above  we arrive at the result claimed.
\end{proof}

Applying Theorem~\ref{intromain} and Lemma~\ref{localcorrelations1} with $g=1$ an we deduce the following corollary.
\begin{corollaryn}[\ref{nair}]
Let $f:\mathbb{N}\to\mathbb{U}$ be a multiplicative function and $P\in\mathbb{Z}[x]$ Then
\[
\frac{1}{x}\sum_{ n\le x}f(P(n))=\prod_{p\le x}M_p(f(P))+O\left(\mathbb{D}_{P}(1,f;\log x;x)+\frac{1}{\log \log x}\right).
\]
\end{corollaryn}
\end{section}
\begin{section}{Corollaries required for further applications}
To state some corollaries required for our future applications we introduce a few notations. We fix two integer numbers $a,b\ge 1.$ For multiplicative functions $f,g:\mathbb{N}\to\mathbb{C}$ such that $\mathbb{D}(1,f;\infty)<\infty,$ $\mathbb{D}(1,g;\infty)<\infty,$   we set 
 $f=1*\theta,$ $g=1*\gamma.$ For $(r,(a,b))=1$ we define 
 \begin{equation}\label{notation}
 G(f;{g};r;x)=G(r,x):=\prod_{p^k\vert|r,\ p\le x}\left(\theta(p^k){\gamma(p^k)}+\delta_b\sum_{i> k}\frac{\theta(p^k){\gamma(p^i)}}{p^{i-k}}+\delta_a\sum_{i> k}\frac{{\gamma(p^k)}\theta(p^i)}{p^{i-k}}\right)\end{equation}
 and  $\delta_{\l}=0$ when $p\vert {\l}$ and $\delta_{\l}=1$ otherwise.
 For $(r,(a,b))>1$ we set
 \[G(r,x):=0.\]
 We remark that in~\eqref{notation} we allow $k=0.$
We can now deduce the following corollary.
 \begin{corollary}\label{linearmain}
Let $f,g:\mathbb{N}\to\mathbb{U}$ be multiplicative functions. Suppose that $\mathbb{D}(1,f;\infty)<\infty,$ $\mathbb{D}(1,g;\infty)<\infty.$ Let  $a,b\ge 1$,\ $c,d$ be integers with $(a,c)=(b,d)=1$ and $ad\ne bc.$ Then,
\begin{align*}
\frac{1}{x}\sum_{n\le x}f(an+c){g(bn+d)}=\sum_{r|ad-bc}\frac{G(f;g;r;x)}{r}+o(1).
\end{align*}
 \end{corollary}
 \begin{proof}
 We note that  $$
\left| \{ n\le x \mid \exists p^k \ge x, p^k\vert an+c  \}\right|\ll \frac{x}{\log x}$$ and thus the contribution of terms with large prime power factors can be absorbed into the error term. We can now apply Theorem~\ref{intromain} (using the same notations) with $P(n)=an+c$ and $Q(n)=bn+d$ and note that $\text{res}(P,Q)=ad-bc,$ $\omega_P(p^k)=1$ for $p\nmid a$ and $\omega_P(p^k)=0$ for $p\vert a,$  $\omega_Q(p^k)=1$ for $p\nmid b$ and $\omega_Q(p^k)=0$ for $p\vert b,$ $p^k\le x.$
We are left to note that
\[F(d_1,d_2)=\frac{1}{[d_1,d_2]}\] 
and the terms coming from small primes $p\le y,$ such that $(r,(a,b))=1$ 
\[G_s(r)=\sum_{\substack{d_1,d_2\ge 1\\ (d_1,d_2)=r\\ (d_1,a)=1\\ (d_2,b)=1\\p|rd_i\Rightarrow p\le y}}\frac{\theta_{s}(d_1)\overline{\gamma_{s}(d_2)}}{[d_1,d_2]}\]
each has an Euler product
 \begin{align*}
 G_s(a):=\prod_{p^k\vert|a,\ p\le y}\left(\theta(p^k){\gamma(p^k)}+\delta_b\sum_{i>k}\frac{\theta(p^k){\gamma(p^i)}}{p^{i-k}}+\delta_a\sum_{i> k}\frac{{\gamma(p^k)}\theta(p^i)}{p^{i-k}}\right)\end{align*}
and  $\delta_{\l}=0$ when $p\vert {\l}$ and $\delta_{\l}=1$ otherwise.
\end{proof}
We will require the following extension of Corollary~\ref{linearmain} to all ``pretentious" functions. 
\begin{corollaryn}[\ref{introlinear}]

Let $f,g:\mathbb{N}\to\mathbb{U}$ be multiplicative functions for which $\mathbb{D(}f,n^{it},\infty),$ $\mathbb{D}(g,n^{iu},\infty)<\infty$, and write $f_0(n)=f(n)/n^{it}$ and $g_0(n)=g(n)/n^{iu}$.
Let  $a,b\ge 1$,\ $c,d$ be integers with $(a,c)=(b,d)=1$ and $ad\ne bc$. As above we have
$$
\frac{1}{x}\sum_{n\le x}f(an+c){g(bn+d))}=M_i(f(P),g(Q),x)\prod_{p\le x}M_p(f_0(P),g_0(Q)) +o(1).
$$
We have
\[M_i(f(P),g(Q),x)\sim\frac{a^{it}b^{iu}x^{i(t+u)}}{1+i(t+u)}\cdot\]
If $p|(a,b)$ then $M_p(f_0(P),g_0(Q))=1.$ 
In general, if $p\nmid(a,b)$ we have
\begin{align*}
M_p(f_0(P),g_0(Q))=\sum_{\substack{0\le i\le k,\\k\ge 0, \\p^k\vert| ad-bc}}\left(\frac{\theta(p^i){\gamma(p^i)}}{p^i}+\delta_b\sum_{j\ge i+1}\frac{\theta(p^i){\gamma(p^j)}}{p^{j}}+\delta_a\sum_{j\ge i+1}\frac{{\gamma(p^i)}\theta(p^j)}{p^{j}}\right)
\end{align*}
and  $\delta_{\l}=0$ when $p\vert \l$ and $\delta_{\l}=1$ otherwise. Here $f_0=1*\theta$ and  $g_0=1*\gamma.$

\end{corollaryn}
\begin{proof}
 We observe $\mathbb{D}(f_0,1,\infty)<\infty$ and $\mathbb{D}(g_0,1,\infty)<\infty$ and let
\[M(x)=\sum_{n\le x}f_0(an+c){g_0(bn+d)}.\]
Corollary~\ref{linearmain} implies
\[M(y)=y\sum_{r\vert ad-bc}\frac{G(f_0;g_0;r;y)}{d}+o(y).\]
Recall that for any $r\ge 1,$ $(r,(a,b))=1$ 
\begin{align*}G(f_0;g_0;r;x)=G(r,x):=\prod_{p^k\vert|r,\ p\le x}\left(\theta(p^k){\gamma(p^k)}+\delta_b\sum_{i> k}\frac{\theta(p^k){\gamma(p^i)}}{p^{i-k}}+\delta_a\sum_{i> k}\frac{{\gamma(p^k)}\theta(p^i)}{p^{i-k}}\right).\end{align*}
Note that $\mathbb{D}(1,f_0,\infty)<\infty$ together with the fact that $\operatorname{Re}{(\theta(p))}\le 0$ imply 
$$-\sum_{p\ge 1}\frac{\operatorname{Re}{(\theta(p))}}{p}<\infty$$
and thus for $y\gg r$ we have 
\[G(r,y)\ll \exp\left(\sum_{p\ge 1}\frac{\operatorname{Re}{(\theta(p))}}{p}+\frac{\operatorname{Re}{({\gamma}(p))}}{p}\right)=O(1).\]
Furthermore, since $\frac{\operatorname{Re}{(\theta(p))}}{p}\le 0$ and $\frac{\operatorname{Re}{(\gamma(p))}}{p}\le 0$ we use~\eqref{iterate1} to estimate
\begin{align}\label{var}
G(r,x)-G(r,y)&=G(r,y)\left[\prod_{y<p\le x}\left(1+\sum_{k\ge 1}\frac{\theta(p^k)}{p^k}+\sum_{k\ge 1}\frac{{\gamma(p^k)}}{p^k}\right)-1\right]\\&\nonumber =G(r,y)\left[\exp\left(\log\sum_{y<p\le x}\left(1+\sum_{k\ge 1}\frac{\theta(p^k)}{p^k}+\sum_{k\ge 1}\frac{{\gamma(p^k)}}{p^k}\right)\right)-1\right]
\\&\nonumber \ll \left|\exp\left(\sum_{y\le p\le x}\frac{ \operatorname{Re}{(\theta(p))}}{p}+\frac{\operatorname{Re}{(\gamma(p))}}{p}\right)\left(1+O\left(\frac{1}{y}\right)\right)-1\right|\\&\nonumber
\ll\left(\sum_{y<p\le x}\frac{1}{p}\right)\ll \log\left(\frac{\log x}{\log y}\right)\cdot
\end{align}
For $(r,(a,b))>1$ we have $G(r,x)=G(r,y)=0$ and~\eqref{var} holds. 
Hence, 
\[\sum_{r\vert ad-bc}\frac{G(r,y)}{r}=\sum_{r\vert ad-bc}\frac{G(r,x)}{r}+O\left(\log\left(\frac{\log x}{\log y}\right)\right)\]
Since
\[M(y)=y\sum_{r|ad-bc}\frac{G(r,y)}{r}+o(y)\]
we have
\[\frac{M(y)}{y}=\frac{M(x)}{x}+O\left(\log\left(\frac{\log x}{\log y}\right)\right).\]
Summation by parts yields
\begin{align*}\label{integral}
\sum_{n\le x}f(an&+c){g(bn+d)}=\sum_{n\ge 1}(an+c)^{it}(bn+d)^{iu}f_0(an+c){g_0(bn+d)}\\&=
\int_{1}^x(ay+c)^{it}(by+d)^{iu}d(M(y))\\&=M(x)(ax+c)^{it}(bx+d)^{iu}-\int_{1}^xM(y)\left[(ay+c)^{it}(by+d)^{iu}\right]'dy\\&
=M(x)(ax+c)^{it}(bx+d)^{iu}-\frac{1}{x}\int_{1}^xM(x)y\left[(ay+c)^{it}(by+d)^{iu}\right]'dy \\&+O\left(\int_{2}^xy\log\left(\frac{\log x}{\log y}\right)\left|\left[(ay+c)^{it}(by+d)^{iu}\right]'\right|dy\right)\\&=
\frac{M(x)}{x}\int_{2}^x(ay+c)^{it}(by+d)^{iu}dy\\&+O\left(\int_{2}^xy\log\left(\frac{\log x}{\log y}\right)\left|\left[(ay+c)^{it}(by+d)^{itu}\right]'\right|dy\right)
\end{align*}
Note,
\[y\left|\left[(ay+c)^{it}(by+d)^{iu}\right]'\right|\ll \frac{y}{ay+c}+\frac{y}{by+d}=O(1),\]
and so the error term is bounded by 
\[\int_{2}^x\log\left(\frac{\log x}{\log y}\right)dy\ll \frac{x}{\log x}=o(x).\]
Since $|(ay+c)^{it}-(ay)^{it}|=O\left(\frac{t}{y}\right),$ we have 
\[\int_{2}^x(ay+c)^{it}(by+d)^{iu}dy=\int_2^x(ay)^{it}(by)^{iu}dy+o(x).\]
Evaluating the last integral and performing simple manipulations with the Euler factors we conclude
\[\sum_{r\vert ad-bc}\frac{G(f_0;g_0;r;x)}{r}=\prod_{p\le x}M_p(f_0(P),g_0(Q)) +o(1)\]
and the result follows.
\end{proof}

\begin{remark}
Let $f_k(n),$ $k=\overline{1,m}$ be multiplicative functions such that $|f_k(n)|\le 1$ and $\mathbb{D}(f_k(n),n^{it_k};\infty)<\infty$ for all $n\in\mathbb{N}.$ Following the lines of the proof one can generalize Corollary~\ref{introlinear} to compute correlations of the form 
\[
\sum_{n\le x}f_1(a_1n+b_1)f_2(a_2n+b_2)\cdot \dots \cdot f_m(a_mn+b_m).\]
\end{remark}

Finally, we will require the following special case of Corrolary~\ref{linearmain}.
\begin{corollary}\label{relation}
Let $f:\mathbb{N}\to\mathbb{U}$ be a multiplicative function such that $\mathbb{D}(1,f;\infty)<\infty,$ $m\in\mathbb{N}.$ Then,
\[\frac{1}{x}\sum_{n\ge 1}f(n)\overline{f(n+m)}=\sum_{r\vert m}\frac{G_0(r)}{r}+o(1)\]
where $f=1*\theta$ and
\[G_0(r):=\prod_{p^k||r}\left(|\theta(p^k)|^2+2\sum_{i> k}\frac{\operatorname{Re}{(\theta(p^k)\overline{\theta(p^i)}}}{p^{i-k}}\right).\]
\end{corollary}
\begin{proof}
We apply Corollary~\ref{linearmain} with $g=\overline{f},$ $a=b=1,$ $d=0,$ $c=m$ and observe 
\begin{align*}
\prod_{p>x}\left(|\theta(p^k)|^2+2\sum_{i> k}\frac{\operatorname{Re}{(\theta(p^k)\overline{\theta(p^i)})}}{p^{i-k}}\right)=\prod_{p>x}\left(1+2\sum_{i\ge 1}\frac{\operatorname{Re}{(\overline{\theta(p^i)})}}{p^{i}}\right)\to 1.
\end{align*}
Hence, the Euler factors
\[G(a):=\prod_{p^k||a,\ p\le x}\left(|\theta(p^k)|^2+2\sum_{i> k}\frac{\operatorname{Re}{(\theta(p^k)\overline{\theta(p^i)})}}{p^{i-k}}\right)\]
converge to 
\[G_0(a):=\prod_{p^k||a}\left(|\theta(p^k)|^2+2\sum_{i>k}\frac{\operatorname{Re}{(\theta(p^k)\overline{\theta(p^i)})}}{p^{i-k}}\right).\]
\end{proof}
Let $f$ be a multiplicative function such that $|f(n)|\le 1$ and $\mathbb{D}(f(n),n^{it}\chi(n);\infty)<\infty$ for some $t\in\mathbb{R}$ where $\chi$ is a primitive character of conductor $q.$ We define $F$ to be the multiplicative function such that
\begin{equation}\label{twist}
 F(p^k) = \begin{cases} f(p^k)\overline{\chi(p^k)}p^{-ikt}, & \mbox{if } p\nmid q \\ 1, & \mbox{if } p\mid q,\end{cases}
 \end{equation}
and
 $$M_p(F,\overline{F};d)=\lim_{x\to\infty}\frac{1}{x}\sum_{n\le x}F_p(n)\overline{F_p(n+d)}.$$
 We are now ready to establish the formula for correlations when $f$ ``pretends"  to be a modulated character.
\begin{theoremn}[\ref{charactercor}]
Let $f:\mathbb{N}\to\mathbb{U}$ be a multiplicative function such that $\mathbb{D}(f(n),n^{it}\chi(n);\infty)<\infty$ for some $t\in\mathbb{R}$ and $\chi$ is a primitive character of conductor $q.$ Then for any non-zero integer $d$ we have 

\[\frac{1}{x}\sum_{n\le x}f(n)\overline{f(n+d)}=\prod_{\substack{p\le x\\ p\nmid q}}M_p(F,\overline{F};d)\prod_{p^{\l\vert| q}}M_{p^{\l}}(f,\overline{f},d)+o(1),\]
where
\[M_{p^{\l}}(f,\overline{f},d) = \begin{cases} 0, & \mbox{if } p^{\l-1}\nmid d \\ 1-\frac{1}{p}, & \mbox{if } p^{\l-1}\vert| d \\ \left(1 -\frac 1p \right)  \sum_{j=0}^{k} \frac{|f(p^j)|^2}{p^j} -  \frac{|f(p^k)|^2}{p^{k}}, & \mbox{if } p^{\l+k}\vert| d\end{cases}\]
for any $k\ge 0$ and if $p^n\vert| d,$ then 
\[M_p(F,\overline{F},d)=1-\frac{2}{p^{n+1}}+\left(1-\frac{1}{p}\right)\sum_{j>n}\left(\frac{F(p^n)\overline{F(p^j)}}{p^j}+\frac{\overline{F(p^n)}F(p^j)}{p^j}\right).\]
In particular, the mean value is $o(1)$ if $q\nmid d\prod_{p\vert q}p.$
\end{theoremn}
\begin{proof}
We partition the sum according to $r,s\ge 1$ such that $r\vert n$ and $\text{rad}(r)\vert q,$ $(n/r,q)=1$ and $s\vert (n+d)$ and $\text{rad}(s)\vert q,$ $((n+d)/s,q)=1.$ Note that $(r,s)\vert d.$ We write $$n=m\cdot\text{lcm}(r,s)+rb(r)$$ such that $sb(s)-rb(r)=d$ for some integers $b(r),b(s).$ The can now be rewritten as
\[\sum_{n\le x}f(n)\overline{f(n+d)}=\sum_{r,s}f(r)\overline{f(s)}\sum_{\substack{m^*\le \frac{x}{\text{lcm}(r,s)}\\ }}f\left(m^*\frac{s}{(r,s)}+b(r)\right)\overline{f\left(m^*\frac{r}{(r,s)}+b(s)\right)}\]
where the inner sum runs over $m^*$ such that
 $$\left(m^*\frac{s}{(r,s)}+b(r),q\right)=1$$ and
$$\left(m^*\frac{r}{(r,s)}+b(s),q\right)=1.$$ We can therefore define the function $f_1$  such that $f_1(p^k)=f(p^k)$ for all primes $p\nmid q$ and $f_1(p^k)=0$ otherwise. In this case, Corollary~\ref{introlinear} implies
 \begin{align}\label{charactersum1}
\sum_{\substack{m^*\le \frac{x}{\text{lcm}(r,s)}\\ }}f\left(m^*\frac{s}{(r,s)}+b(r)\right)&\overline{f\left(m^*\frac{r}{(r,s)}+b(s)\right)}\\&\nonumber=\sum_{\substack{m\le \frac{x}{\text{lcm}(r,s)}\\ }}f_1\left(m\frac{s}{(r,s)}+b(r)\right)\overline{f_1\left(m\frac{r}{(r,s)}+b(s)\right)}
  \end{align}
  where now $m$ runs over all integers up to $\frac{x}{\text{lcm}(r,s)}.$
 We can now factor $f_1(n)=\chi(n)F(n).$ Note $\mathbb{D}(F,1,\infty)< \infty.$ Let $m=kq+a$ where $a$ runs over residue classes $\text{mod}(q).$ The sum in~\eqref{charactersum1} can be rewritten
 \begin{align*} \sum_{r,s}f(r)g(s)&\sum_{a\ \text{mod}(q)}\chi\left(a\frac{s}{(r,s)}+b(r)\right)\overline{\chi\left(a\frac{r}{(s,r)}+b(s)\right)}\\&\times\sum_{k\le \frac{x}{q\text{lcm}(r,s)} }F\left(kq\frac{s}{(r,s)}+a\frac{s}{(r,s)}+b(r)\right)\overline{F\left(kq\frac{r}{(r,s)}+a\frac{r}{(r,s)}+b(s)\right)}.
 \end{align*}
We apply Corollary~\ref{introlinear} to the inner sum and  observe that
$$a_2b_1-a_1b_2=\frac{dq}{(r,s)}$$ and the asymptotic in Corollary~\ref{introlinear}  does not depend on $b_1,b_2$ and consequently on the residue class $a(\text{mod}(q)).$ Hence, up to a small error the innermost sum is equal to 
\[\sum_{m\le \frac{x}{q[s,r]}}F\left(m\frac{s}{(r,s)}+b(r)\right)\overline{F\left(m\frac{r}{(r,s)}+b(s)\right)}.\]
We now focus on the sum
\begin{equation}\label{charactersum2}
\sum_{a\ \text{mod}(q)}\chi\left(a\frac{s}{(r,s)}+b(r)\right)\overline{\chi\left(a\frac{r}{(s,r)}+b(s)\right)}.\end{equation}
Let $q=p_1^{a_1}p_2^{a_2}...p_k^{a_k}$ and $\chi=\chi_{p_1^{a_1}}\chi_{p^{a_2}}\cdot ... \cdot \chi_{p_k^{a_k}},$ where each $\chi_{p_i^{a_i}}$ is a primitive character of conductor $p_i^{a_i}.$ By the Chinese Reminder Theorem the sum~\eqref{charactersum2} equals 
\begin{align*}
\sum_{a\ \text{mod}(q)}\chi\left(a\frac{s}{(r,s)}+b(r)\right)&\overline{\chi\left(a\frac{r}{(s,r)}+b(s)\right)}\\&\nonumber=\prod_{p^k\vert| q}\sum_{a\ \text{mod}(p^k)}\chi_{p^k}\left(a\frac{s}{(r,s)}+b(r)\right)\overline{\chi_{p^k}\left(a\frac{r}{(s,r)}+b(s)\right)}.\end{align*}
We claim that the last sum is zero unless $r=s.$ Indeed, if $r\ne s,$ then there exists prime $p$ such that $p^i\vert |r$ and $p^j\vert |s$ for $j>i.$ Since $(r/(r,s),p)=1$ we can make change of variables $$a\to \frac{ar}{(r,s)}\text(mod(p^k))$$ and the $p-$th factor can rewritten
  \[\sum_{a\ \text{mod}(p^k)}\chi_{p^k}(ap^{j-i}t+b_1(r))\overline{\chi_{p^k}(a+b_1(s))}\]
  where $(t,p)=1.$ If $j-i\ge k,$ then the first term is fixed and the second runs over all residues modulo $p^k.$ So the sum is zero. 
  If $j-i<k,$ we write $a=A+p^{k-(j-i)}L$ where $A$ runs over residues  $\text{mod}(p^{k-(j-l)})$ and $L$ runs over residues modulo $p^{j-i}.$
  Then our sum becomes
  \[\sum_{A\ \text{mod}(p^{k-(j-l)})}\chi_{p^k}(Ap^{j-i}t+b_1(r))\sum_{L\ \text{mod}p^{j-i}}\overline{\chi_{p^k}(A+b_1(s)+p^{k-j+i}L)}\]

It is easy to show that the inner sum
\[\sum_{L\ \text{mod}p^{j-i}}\overline{\chi(A+b_1(s)+p^{k-j+i}L)}=0.\]
Thus the main contribution comes from the terms $r=s=R.$ In this case we have $R(b(s)-b(r))=d=bR$ and we can take $b(r)=0,$ $b(s)=b.$ Our character sum then can be rewritten as
\[\sum_{a\ \text{mod}(q)}\chi(a)\overline{\chi(a+b)}.\]
To evaluate the last sum, we split it into prime powers. Now if $p^k\vert| q$ and $p^{i}\vert| b$ (possibly $i=0$) then we have nonzero contribution if and only if $i\ge k-1.$ Indeed, let $b=p^{i}b_1,$
$(b_1,p)=1.$ We note  

\[\sum_{a\ \text{mod}(p^k)}\chi_{p^k}(a)\overline{\chi_{p^k}(a+b)}=\sum_{\substack{c\ \text{mod}(p^k),\\ (c,p)=1}}\chi_{p^k}(p^ic+1).\]
This sum is $0$ if $i\le k-2$ and equals to $-p^{k-1}$ whenever 
$i=k-1$ and $\phi(p^k)$ whenever $i\ge k.$
We thus have 
\[\sum_{a\ \text{mod}(q)}\chi(a)\overline{\chi(a+b)}=\prod_{\substack{p^k\vert| q\\ p^i\vert| b\\ i\le k-1}}\mu(p^{k-i})p^i\prod_{\substack{p^k\vert| q\\ p^{k}\vert b}}\phi(p^{k})\] and the result follows by combining this with Corollary~\ref{introlinear} and easy manipulations with the Euler products.
\end{proof}
Combining the last proposition with Corollary~\ref{relation} we deduce 
\begin{corollary}\label{keytotao}
Let $f$ be a multiplicative function such that $|f(n)|\le 1,$ $\mathbb{D}(f(n),n^{it}\chi(n);\infty)<\infty$ for some primitive character $\chi$ of conductor $q.$ Then
\begin{align*}
\frac{1}{x}\sum_{n\le x}f(n)\overline{f(n+1)}&=\frac{\mu(q)}{q}
\prod_{\substack{p\ge 1\\ p\nmid q}}\left(2\operatorname{Re}\left(1-\frac{1}{p}\right)\left(\sum_{k\ge 0}\frac{f(p^k)\overline{\chi(p^k)}p^{-ikt}}{p^k}\right)-1\right)+o(1).
\end{align*}
\end{corollary}
We remark that using the same arguments one may establish the formula for the correlations
\[\sum_{n\le x}f(n)g(n+m)\]
for  $\mathbb{D}(f(n),n^{it_1}\chi(n),\infty)<\infty$ and $\mathbb{D}(g(n),n^{it_2}\psi(n),\infty)<\infty.$ We state here one particular case when $m=1.$
\begin{prop}
Let $f,g:\mathbb{N}\to\mathbb{U}$ be two multiplicative functions such that  $\mathbb{D}(f(n),n^{it_1}\chi(n),\infty)<\infty$ and $\mathbb{D}(g(n),n^{it_2}\psi(n),\infty)<\infty$ for some primitive characters $\chi,\psi.$ Let $R=\frac{q_{\psi}}{(q_{\chi},q_{\psi})}$ and $S=\frac{q_{\chi}}{(q_{\chi},q_{\psi})},$ $Q=[q_{\chi},q_{\psi}].$ Then
\begin{align*}
\frac{1}{x}\sum_{n\le x}f(n)&g(n+1)=\frac{R^{it_1}S^{it_2}}{i(t_1+t_2)+1}f(R)g(S)\sum_{a\ \text{mod}(Q)}\chi(aS+b(R))\psi(aR+b(S))\\&\times
\prod_{\substack{p\le x\\ p\nmid Q}}\left(\left(1-\frac{1}{p}\right)\left(\sum_{k\ge 0}\frac{f(p^k)p^{-ikt_1}}{p^k}\right)+\left(1-\frac{1}{p}\right)\left(\sum_{k\ge 0}\frac{{g(p^k)}p^{-ikt_2}}{p^k}\right)-1\right)+o(1).
\end{align*}
\end{prop}
\begin{proof} We follow the lines of the proof of Proposition~\ref{charactercor} and note that in this case $(r,s)=1$ and the only term that contributes is 
\[r=R=\frac{q_{\psi}}{(q_{\chi},q_{\psi})}\] and $$s=S=\frac{q_{\chi}}{(q_{\chi},q_{\psi})}\cdot$$
\end{proof} 
\end{section}
\begin{section}{Application to the Erd\H{o}s-Coons-Tao conjecture}
In this sections we are going to study multiplicative functions $f:\mathbb{N}\to\mathbb{T},$ such that
\begin{equation}\label{boundedsum1}
\limsup_{x\to\infty}|\sum_{n\le x}f(n)|<\infty.
\end{equation}
We first focus on the complex valued case and the proof of Theorem~\ref{complexcor}.
The key tool is the following recent result by Tao~\cite{MR4}.
\begin{theorem}\label{averagedelliot}{\bf [Tao]}
Let $a_1,a_2$ be natural numbers, and let $b_1,b_2$ be integers such that $a_1b_2-a_2b_1\ne 0.$ Let $\varepsilon>0,$ and suppose that $A$ is sufficiently large depending on $\varepsilon,a_1,a_2,b_1,b_2.$ Let $x\ge\omega\ge A,$ and let $g_1,g_2:\mathbb{N}\to\mathbb{U}$ be multiplicative functions with $g_1$ non-pretentious in the sense that
\[\sum_{p\le x}\frac{1-\operatorname{Re}(f(p)\chi(p)p^{it})}{p}\ge A\]
for all Dirichlet character $\chi$ of period at most $A,$ and all real numbers $|t|\le Ax.$ Then
\[\left| \sum_{x/\omega<n\le x}\frac{g_1(a_1n+b_1)g_2(a_2n+b_2)}{n}\right|\le \varepsilon \log{\omega}.\]
\end{theorem}
We will require the following technical lemma due to Elliott (Lemma $17$ in \cite{MR2658182}).
\begin{lemma}\label{pretenditerations}{\bf [Elliott]}
Let $|g(p)|\le 1$ on the primes, and 
\[\sum_{p\le x}\frac{1-\operatorname{Re}(g(p)p^{it\lambda(x)})}{p}\ll 1\]
for all $x\ge 2.$ Suppose that $\lambda(x)\ll x$ uniformly for all sufficiently large $x.$ Then there exists a constant $C$ such that 
$\lambda(x)-C\ll \frac{1}{\log x}$ and the series 
 \[\sum_{p\ge 1}\frac{1-\operatorname{Re}(g(p)p^{itC})}{p}<\infty.\]
\end{lemma}

\begin{lemma}\label{generaltao}
Suppose that for a multiplicative $f:\mathbb{N}\to\mathbb{T},$~\eqref{boundedsum1} holds. Then there exists a primitive character $\chi$ and $t\in\mathbb{R},$ such that $\mathbb{D}(f(n),\chi(n)n^{it},\infty)<\infty.$
\end{lemma}
\begin{proof}
Let $H\in\mathbb{N}.$
Suppose that for each $1\le h\le H$ we have
\[\frac{1}{\log x}\sum_{n\le x}\frac{f(n)\overline{f(n+h)}}{n}\le \frac{1}{2H}\cdot\]
Consider
\[T(x):=\frac{1}{\log x}\sum_{n\le x}\frac{1}{n}\left|\sum_{k=n+1}^{n+H+1}f(k)\right|^2\]
Expanding the square we get
\[T(x)=\sum_{1\le h_1\ne h_2\le H}\frac{1}{\log x}\sum_{n\le x}\frac{f(n+h_1)\overline{f(n+h_2)}}{n}\cdot\]
The diagonal contribution $h_1=h_2$ is $1+o(1).$ For $h_2>h_1$ we introduce $h=h_2-h_1$ and replace $n$ in the denominator by $N=n+h_1$ at a cost $\ll H/\log x.$ We change the range for $N$ from $1+h_1\le N\le x+h_1$ to $1\le n\le x$ at a cost
of $\ll \log H/\log x.$ Therefore
\begin{align*}
T(x)&=H+o(1)-\sum_{|h|\le H}(H-|h|)\cdot\frac{1}{\log x}\sum_{N\le x}\frac{f(N)\overline{f(N+h})}{N}\\&\ge H-(H^2-H)\cdot\frac{1}{2H}+o(1)= \frac{H}{2}+O(1)
\end{align*}
for $x\to\infty.$ This contradicts~\eqref{boundedsum1} for sufficiently large $H\ge 1.$ Thus, for a fixed $H\ge 1,$ and large every large $x\gg 1,$ there exists $1\le h_x\le H$\
such that
\[\frac{1}{\log x}\sum_{n\le x}\frac{f(n)\overline{f(n+h_x)}}{n}\gg 1.\]
Since $h_x\le H,$ we can apply Theorem~\ref{averagedelliot} to conclude that there exists $A=A(H)\ge 0$ such that for any sufficiently large $x,$ there exists $t_x\in\mathbb{R},$ $|t_x|\le Ax$ and a primitive character $\chi$ of modulus $D\le A,$ such that $\mathbb{D}(f(n),n^{it_x}\chi(n);x)\le A,$ namely
\[\sum_{p\le x}\frac{1-\operatorname{Re}(f(p)p^{-it_x}\overline{\chi(p)})}{p}\le A.\]
Since we have only finitely many possibilities for $\chi,$ there exists $k\in\mathbb{N},$ such that $\chi^k(p)=1$ for all primes $p\ge A.$ Triangle inequality now implies
\[kA\ge k\mathbb{D}(f(n),n^{it_x}\chi(n);x)\ge \mathbb{D}(f^k(n),n^{ikt_x}\chi^k(n);x)+O(1)= O(1)+\mathbb{D}(f^k(n),n^{ikt_x};x).\]
Hence, $\mathbb{D}(f^k(n),n^{ikt_x};x)=O(1).$ Clearly $|kt_x|\le kA$ and Lemma~\ref{pretenditerations} implies that there exists fixed $t_1>0$ such that $$kt_x=t_1+O\left(\frac{1}{\log x}\right).$$ Since
\[\mathbb{D}(f(n),n^{i\frac{t_1}{k}}\chi(n);x)\le \mathbb{D}(f(n),n^{it_x}\chi(n);x)+O(1)=O(1)\] 
and we have only finitely many choices of primitive characters $\chi(n),$ this implies that there exists unique primitive character $\chi_1(n)$ and $t_{\chi_1}=\frac{t_1}{k}$ such that 
$\mathbb{D}(f(n),p^{it_{\chi_1}}\chi_1(n);\infty)<\infty$ and the result follows.
\end{proof}
We now refine the result of Lemma~\ref{generaltao}.
\begin{theoremn}[\ref{complexcor}]
Suppose for a multiplicative $f:\mathbb{N}\to\mathbb{T},$~\eqref{boundedsum1} holds. Then there exists a primitive character $\chi$ of an odd conductor $q$ and $t\in\mathbb{R},$ such that $\mathbb{D}(f(n),\chi(n)n^{it};\infty)<\infty$ and $f(2^k)=-\chi^k(2)2^{-ikt}$ for all $k\ge 1.$
\end{theoremn}
\begin{proof}
Applying Lemma~\ref{generaltao}, we can find a primitive character $\chi$ of conductor $q$ and $t\in\mathbb{R}$ such that $\mathbb{D}(f(n),\chi(n)n^{it};\infty)<\infty.$
Theorem~\ref{charactercor} implies that for any $d\ge 0,$ we have
\[S_d=\lim_{x\to\infty}\frac{1}{x}\sum_{n\le x}f(x)\overline{f(x+d)}=\prod_{\substack{p\le x\\ p\nmid q}}M_p(F,\overline{F};d)\prod_{p^{\l}\vert| q}M_{p^{\l}}(f,\overline{f},d).\]
For fixed $H\ge 1,$ we can now write
\begin{align*}
\lim_{x\to\infty}\frac{1}{x}\sum_{n\le x}\left|\sum_{k=n+1}^{n+H+1}f(k)\right|^2&=\lim_{x\to\infty}\frac{1}{x}\left[\sum_{h=0,\ n\le x}Hf(n)\overline{f(n+h)}+ 2\sum_{1\le h\le H}(H-h)\sum_{n\le x}f(n)\overline{f(n+h)}\right]\\&=HS_0+2\sum_{h=1}^H(H-h)S_h=H+2\sum_{N=1}^{H-1}\sum_{n=1}^NS_m.
\end{align*}
We note that all $S_m\le 1$ and Theorem~\ref{charactercor} implies that each $S_m$ behaves like a scaled multiplicative function, since it is given by the Euler product. We are going to show that there exists 
$\lim_{N\to\infty}\frac{1}{N}\sum_{n\le N}S_n=c$ and so
\[H+2\sum_{N=1}^{H-1}\sum_{n=1}^NS_m=O(1)\sim H+2\sum_{N=1}^Hcn=cH^2+O(H).\]
Latter would imply that $c=0.$ We turn to the computations of the corresponding mean values. Clearly
\[\lim_{N\to\infty}\frac{1}{N}\sum_{n\le N}S_n=\prod_{p\le N}S(p)\]
where $S(p)$ denotes the local factor that corresponds to prime $p.$ If $p\nmid q,$ then using Theorem~\ref{charactercor} and simple computations
\begin{align*}
S_p=\sum_{k\ge 0}\left(\frac{1}{p^k}-\frac{1}{p^{k+1}}\right)M_p(F,\overline{F},p^k)=\left|\left(1-\frac{1}{p}\right)\sum_{k\ge 0}\frac{F(p^k)}{p^k}\right|^2.
\end{align*}
If $p^{\l}\vert| q,$ then again using Theorem~\ref{charactercor} we get
 \begin{align*}
S_p=\sum_{k\ge 0}\left(\frac{1}{p^k}-\frac{1}{p^{k+1}}\right)M_{p^{\l}}(f,\overline{f},p^k)=\frac{1}{p^{\l-1}}\left(1-\frac{1}{p}\right)^2.
\end{align*}
Since $c=0,$ one of the Euler factors has to be $0.$ The only possibility then is $S_2=0$ and $2\nmid q$ and $F(2^k)=-1$ for all $k\ge 1.$ This completes the proof.
\end{proof}
\noindent \textbf{Proof of the Erd\H{o}s-Coons-Tao conjecture.}
We now move on to the proof of Theorem~\ref{tao}.  
It turns out that periodic multiplicative functions with zero mean have the following equivalent characterization that we will use throughout the proof.
\begin{prop}\label{equiv}
Suppose that $f$ multiplicative with each $|f(n)|\le 1$ for all $n\in\mathbb{N}.$ Then there exists an integer $m$ such that $f(n+m)=f(n)$ for all $n\in\mathbb{N}$ and $\sum_{n=1}^m f(n)=0$ if and only if $f(2^k)=-1$ for all $k\ge 1$ and there exists an integer $M$ such that if prime power $p^k\ge M$ then $f(p^k)=f(p^{k-1}).$
\end{prop}
 \begin{proof}
Suppose that $f(n+m)=f(n)$ for all $n\ge 1$ and $\sum_{n=1}^m f(n)=0.$  From periodicity we have $f(km)=f(m)$ for all $k\ge 1,$ and so if $p^a\vert| m$ then $f(p^b)=f(p^a)$ for all $b\ge a.$
In particular if $p$ does not divide $m$ then $f(p^b)=1.$
Hence, \[\sum_{n=1}^m f(n) = \sum_{d|m}  f(d) \phi\left(\frac{m}{d}\right)
 = \prod_{p^a\vert| m} \left(p^a\left(1-\frac{1}{p}\right) \left( \sum_{1\le k\le a-1}\frac{f(p^k)}{p^k}\right) + f(p^a) \right).\]
Consequently, some factor has to be $0.$ The only possibility is then $p=2$ and $f(2^k)=-1$ for all $k\ge 1.$ The other direction immediately follows from the Chinese remainder theorem.  
 \end{proof}
Our starting point is the following result:
 \begin{theorem}\label{taopret}{\bf[Tao, 2015] }
If for a multiplicative $f:\mathbb{N}\to \{-1,1\}$
\[\limsup_{x\to\infty}|\sum_{n\le x}f(n)|<\infty,\] then 
$f(2^j)=-1$ for all $j\ge 0$ and 

\[\sum_{p}\frac{1-f(p)}{p}<\infty.\]

\end{theorem}
In what follows we restrict ourselves to the multiplicative functions $f:\mathbb{N}\to\{-1,1\}$ such that $\mathbb{D}(1,f,\infty)<\infty,$ $f=1*g$ and $f(2^{j})=1$ for all $j\ge 1.$ For such such functions we are going to drop the subscript and set
 \begin{equation}\label{quant}
 G_0(a)=G(a):=\prod_{p^k||a}\left(|g(p^k)|^2+2\sum_{i\ge k+1}\frac{g(p^k){g(p^i)}}{p^{i-k}}\right).
 \end{equation}
The following lemma summarizes properties of $G(a)$ that we will use throughout the proof.
\begin{lemma}\label{properties}
Let $G(a)$ be as above. Then
\begin{enumerate}
\item $G(4a)=0,$ $a\in\mathbb{N};$
\\
\item $G(2a)=-4G(a)$ for odd $a$;
\\
\item $\sum_{a\ge 1}\frac{G(a)}{a^2}=0;$
\\
\item If $f(3)=1,$ then $G(a)\le 0$ for all odd $a;$
\\
\item $\sum_{a\ge 1}\frac{G(a)}{a}=1.$
\end{enumerate}
\end{lemma}
\begin{proof}
Note that $g(2)=-2$ and $g(2^{i})=f(2^i)-f(2^{i-1})=0$ for $i\ge 2.$ Thus $G(4a)=0$ and $G(2a)=-4G(a)$ for odd $a.$ The third part immediately follows from 
\begin{align*}
\sum_{a\ge 1}\frac{G(a)}{a^2}=\sum_{a\ge 1,\ a\ odd}\frac{G(a)}{a^2}+\sum_{a\ge 1,\ a\ odd}\frac{G(2a)}{(2a)^2}=0.
\end{align*}
To prove $(4),$ fix $p$ and suppose $p^k\vert|a.$ We note that for $k=0,$ the Euler factor 
\[E_p(a)=1+2\sum_{i\ge 1}\frac{g(p^i)}{p^i}\ge 1-\frac{4}{p-1}\ge 0\] for $p\ge 5.$ Note $E_2(a)=1-2=-1.$ 
If $3^0\vert| a,$ then $g(3)=f(3)-1=0$ and $E_3(a)\ge 1-\frac{4}{9}\cdot\frac{3}{2}=\frac{1}{3}>0.$ Suppose that $p^k\vert|a $ and $k\ge 1.$ Then,
\[E_p(a)=|g(p^k)|^2+2\sum_{i\ge k+1}\frac{g(p^k){g(p^i)}}{p^{i-k}}\ge 4- \frac{8}{p-1}\ge 0\]
for $p\ge 3.$ Hence the only negative Euler factor is $E_2$ and $(4)$ follows. To prove $(5),$ we take $m=0$ in Corollary~\ref{relation} to arrive at
\[\lim_{x\to\infty}\frac{1}{x}\sum_{n\le x}f(n)\overline{f(n+0)}=1=\sum_{a|0}\frac{G(a)}{a}=\sum_{a\ge 1}\frac{G(a)}{a}\cdot\]
\end{proof}
\begin{lemma}
Suppose $G(a)\ne 0.$ Then,
\[G(a)\gg  \left(\frac{5}{4}\right)^{\omega(a)-1}\cdot \frac{2}{5}\cdot G(1).\]
\end{lemma}
\begin{proof}
Recall,
\begin{align*}
G(a)=\prod_{p^k||a}\left(|g(p^k)|^2+2\sum_{i\ge k+1}\frac{g(p^k){g(p^i)}}{p^{i-k}}\right).
\end{align*}
Note $g(p^k)g(p^{k+1})\le 0$ and so if  $p^k\vert|a $ and $k\ge 1$ we have
\[E_p(a)=|g(p^k)|^2+2\sum_{i\ge k+1}\frac{g(p^k){g(p^i)}}{p^{i-k}}\ge 4- \frac{8}{p}\cdot\frac{1}{1-\frac{1}{p^2}}=4-\frac{8p}{p^2-1}\cdot\]
For $p=3$ the last bound reduces to $E_3(a)\ge 1$ and for $p\ge 5$ we clearly have $E_p(a)\ge 2.$ 
 For $k=0,$ we have
\[E_p(1)=1+2\sum_{i\ge 1}\frac{g(p^i)}{p^i}\le 1+\frac{4}{p}\cdot\frac{1}{1-\frac{1}{p^2}}=1+\frac{4p}{p^2-1}.\]
Consequently, for $k\ge 1$ and $p>3$
\[E_p(p^k)\ge \frac{5}{4}E_p(1).\]
Taking into account $p=3$ we conclude
\[|G(a)|=\left|\prod_{p^k||a,\ k\ge 1}\left(|g(p^k)|^2+2\sum_{i\ge k+1}\frac{g(p^k){g(p^i)}}{p^{i-k}}\right)\right|\ge \left(\frac{5}{4}\right)^{\omega(a)-1}\cdot \frac{2}{5}\cdot |G(1)|.\]
\end{proof}
In fact, it is easy to check that $G(1)\ne 0$ and thus the last lemma provides nontrivial lower bound for $G(a).$
In the next lemma we compute the second moment of the partial sums over the interval of fixed length.
\begin{lemma}\label{correlation}
Let $H\in\mathbb{N}.$ Then
\[\frac{1}{x}\sum_{n\le x}\left(\sum_{k=n+1}^{n+H+1}f(k)\right)^2=-2\sum_{a\ge 1,\ a\ odd}G(a)\left\|\frac{H}{2a}\right\|+o_{x\to\infty}(1).\]
\end{lemma}
\begin{proof}
Note
\begin{align*}
\frac{1}{x}\sum_{n\le x}\left(\sum_{k=n+1}^{n+H+1}f(k)\right)^2&=\frac{1}{x}\left[\sum_{h=0,\ n\le x}Hf(n)f(n+h)+ 2\sum_{1\le h\le H}(H-h)\sum_{n\le x}f(n)f(n+h)\right]+o(1)\\&=\sum_{a\ge 1}\frac{G(a)}{a}\left(H+2\sum_{\substack{1\le h\le H,\\ a| h}}(H-h)\right)+o_{x\to\infty}(1)
\end{align*}
To compute the corresponding coefficient we write $H=ra+s,$ $0\le s<a$ to arrive at
\begin{align*}
ra+s+2\sum_{1\le m\le r}(ra+s-ma)&=ra+s+ar(r-1)+2rs\\&=\frac{(ra+s)^2}{a}+a\left(\frac{s}{a}-\left(\frac{s}{a}\right)^2\right).
\end{align*}
Plugging this into our formula  and using $(4),$ $(1),$ $(2)$ from the Lemma~\ref{properties} we get
\begin{align*}
H^2\sum_{a\ge 1}\frac{G(a)}{a^2}&+\sum_{a\ge 1}G(a)\left(\left\{\frac{H}{a}\right\}-\left\{\frac{H}{a}\right\}^2\right)=\sum_{a\ge 1}G(a)\left(\left\{\frac{H}{a}\right\}-\left\{\frac{H}{a}\right\}^2\right)\\&=\sum_{a\ge 1,\ a\ odd}G(a)\left[\left(\left\{\frac{H}{a}\right\}-\left\{\frac{H}{a}\right\}^2\right)-4\left(\left\{\frac{H}{2a}\right\}-\left\{\frac{H}{2a}\right\}^2\right)\right]\\&=-2\sum_{a\ge 1,\ a\ odd}G(a)\left\|\frac{H}{2a}\right\|,
\end{align*}
since 
\[\left(\left\{\frac{H}{a}\right\}-\left\{\frac{H}{a}\right\}^2\right)-4\left(\left\{\frac{H}{2a}\right\}-\left\{\frac{H}{2a}\right\}^2\right)=-2\left\|\frac{H}{2a}\right\|.\]
\end{proof}
We are now ready to prove Theorem~\ref{tao}.
\begin{theoremn}[\ref{tao}]
Let $f:\mathbb{N}\to\{-1,1\}$ be a multiplicative function. Then
\begin{align*}\limsup_{x\to\infty}\left|\sum_{n\le x}f(n)\right|<\infty,\end{align*} if and 
only if there exists an integer $m\ge 1$ such that $f(n+m)=f(n)$ for all $n\ge 1$ and $\sum_{n=1}^m f(n)=0.$
\end{theoremn}
\begin{proof}
If $f$ satisfies $\sum_{i=1}^mf(i)=0$ and $f(n)=f(n+m)$ for some $m\ge 1,$
then for all $x\ge 1,$
$$\left|\sum_{n\le x}f(n)\right|\le m$$
and the claim follows.
In the other direction, we assume $|\sum_{n\le x}f(n)|=O_{x\to\infty}(1).$ By Theorem~\ref{taopret} we must have $f(2^i)=-1$ for all $i\ge 1$ and 
$\mathbb{D}(1,f,\infty)<\infty.$
By the Lemma~\ref{correlation} we must have that for all $H\ge 1,$
\[\frac{1}{x}\sum_{n\le x}\left(\sum_{k=n+1}^{n+H+1}f(k)\right)^2=-2\sum_{a\ge 1,\ a\ odd}G(a)\left\|\frac{H}{2a}\right\|+o_{x\to\infty}(1)=O_{x\to\infty}(1).\]

Suppose that there is an infinite sequence of odd numbers $\{a_n\}_{n\ge 1}$ such that $g(a_n)\ne 0.$ Observe, $|G(a_n)|\gg 1.$ Choose $H=\text{lcm}[a_1,\dots a_M].$ If $f(3)=1,$ then by Lemma~\ref{properties}, part $(4)$ we have 
\[-2\sum_{a\ge 1,\ a\ odd}G(a)\left\|\frac{H}{2a}\right\|\ge -2\sum_{1\le n\le M}G(a_n)\left\|\frac{H}{2a_n}\right\|\gg M.\] 
This is clearly impossible if $M$ is sufficiently large.

Suppose $f(3)=-1.$ Let 
\[G^*(a)=\prod_{p^k\vert|a,\ p>3}\left(|g(p^k)|^2+2\sum_{i\ge k+1}\frac{g(p^k){g(p^i)}}{p^{i-k}}\right)\] 
and 
\[S(H)=-2\sum_{a\ge 1,\ (a,6)=1}G^*(a)\left\|\frac{H}{2a}\right\|.\]
Note that 
\begin{equation}\label{recursion}
-2\sum_{a\ge 1,\ a\ odd}G(a)\left\|\frac{H}{2a}\right\|=\sum_{i\ge 0}E_3\left(3^i\right)S\left(\frac{H}{3^i}\right)=O(1).
\end{equation}
If $E_3(1)\ge 0$ then we proceed as in the previous case. If $E_3(1)<0,$ then $g(3)=f(3)-1=-2.$ Since $g(p^k)g(p^{k+1})\le 0$ for all $k\ge 0$ we get 
\[E_3(3)\ge 4-\frac{8}{9}\cdot\frac{1}{1-\frac{1}{9}}\ge 3\] 
and
\[0>E_3(1)=1+2\sum_{i\ge 1}\frac{g(3^i)}{3^i}\ge 1-\frac{4}{3}\cdot\frac{1}{1-\frac{1}{9}}=-\frac{1}{2}.\]
Since $E_3(3^k)\ge 0$ for all $k\ge 1,$ applying triangle inequality in~\eqref{recursion} yields  
\begin{equation}\label{iterate}
S(H)\ge \frac{E_3(3)S\left(\frac{H}{3}\right)}{-E_3(1)}+O(1)\ge 6S\left(\frac{H}{3}\right)-M.
\end{equation}
If there is an infinite sequence $\{b_n\}_{n\ge 1}$ such that $g(b_n)\ne 0$ and $(b_n,6)=1,$ then we  
select $H_0$  as before such that $S(H_0)\ge M$ and $S(3H_0)\ge M.$ Then~\eqref{iterate} yields 
$S(3H_0)\ge 5S(H_0).$
By induction one easily gets that for all $n\ge 1,$ 
\[S(3^nH_0)\ge 5^nS(H_0).\]
This implies, that for the sequence $H_n=3^nH_0$ we have $S(H_n)\gg H_n^{1+c}.$ From the other hand
\[\sum_{a\ge H,\ (a,6)=1}\frac{G^*(a)}{a}=o_{H\to\infty}(1)\] and so
\begin{align*}
S(H)=-2\sum_{a\ge 1,\ (a,6)=1}G^*(a)\left\|\frac{H}{2a}\right\|&\ll \sum_{a\le H,\ (a,6)=1}G^*(a)+H\sum_{a\ge H,\ (a,6)=1}\frac{G^*(a)}{a}\\&\ll \sqrt{H}\sum_{a\le \sqrt{H},\ (a,6)=1}\frac{G^*(a)}{a}+H\sum_{\sqrt{H}\le a\le H,\ (a,6)=1}\frac{G^*(a)}{a}+o(H)
\end{align*}
and so $S(H)=o(H).$

To finish the proof we are left to handle the case $g(3^k)\ne 0$ for infinitely many $k\ge 1$ and there exists finitely many $b_1,b_2\dots,b_m$ $(b_i,6)=1,$ $i\ge 1$ and $g(b_i)\ne 0.$ In this case we have 
\[S(H)\le \sum_{i=1}^m G^*(b_i):=M.\]  Choose $H_0=\text{lcm}[b_1,\dots,b_m]$ and observe that $S(3^kH_0)\ge M/2$ for $k=1,\dots K.$ Then,
\begin{align*}
-2\sum_{a\ge 1,a\ odd}G(a)\left\|\frac{3^KH_0}{2a}\right\|&=\sum_{i\ge 0}E_3\left(3^i\right)S\left(\frac{3^KH_0}{3^i}\right)\\&\ge \sum_{1\le i\le K}E_3\left(3^i\right)S\left(\frac{3^KH_0}{3^i}\right)-E_3(1)S(H_0)\\&\ge \frac{M}{2}\sum_{1\le i\le K}E_{3}(3^k)-M.
\end{align*}
The last sum is bounded if $E_3(3^k)=0$ for all $k\ge K_0.$ Consequently, $f(3^{k})=f(3^{k+1})$ for $k\ge K_0$ and the result follows.
\end{proof}
\end{section}
\begin{section}{Applications to the conjecture of K\'atai}
Let $f:\mathbb{N}\to \mathbb{C}$ be a multiplicative function
and $\triangle f(n)=f(n+1)-f(n).$ In this section we focus on proving
 \begin{theoremn}[\ref{introkatai}]
If $f:\mathbb{N}\to \mathbb{C}$ is a multiplicative function and
\begin{equation}\label{kataidifference} 
\lim_{x\to\infty}\frac{1}{x}\sum_{n\le x}|\triangle f(n)|=0
\end{equation}
then either 
\[\lim_{x\to\infty}\frac{1}{x}\sum_{n\le x}|f(n)|=0\] 
or
$f(n)=n^{s}$ for some $\operatorname{Re}(s)<1.$
\end{theoremn}

In~\cite{MR1764803}, K\'atai, building on the ideas of Maclauire and Murata~\cite{MR603061}, showed that in order to prove Theorem~\ref{introkatai}, it is enough to consider multiplicative $f,$ with $|f(n)|=1$ for all $n\ge 1.$ 
Observe, that if we denote
\[S(x)=\frac{1}{x}\sum_{n\le x}|\triangle (n)|\] then~\eqref{kataidifference} implies

 \[\sum_{n\le x}\frac{|\triangle f(n)|^2}{n}\le \sum_{n\le x}\frac{2|\triangle f(n)|}{n}\ll \int_{1}^x\frac{S(t)}{t^2}dt=o_{x\to\infty}(\log x).\]
We begin by proving the following lemma.
\begin{lemma}\label{kataireduction}
Suppose that $f:\mathbb{N}\to \mathbb{T}$ is multiplicative and
\[\sum_{n\le x}\frac{|\triangle f(n)|^2}{n}\le 2(1-\varepsilon)\log x\]
for $x$ sufficiently large and some $\varepsilon<1.$
 Then, there exists a primitive character $\chi_1(n)$ and $t_{\chi_1}\in\mathbb{R}$ such that $\mathbb{D}(f(n),\chi_1(n)n^{it_{\chi_1}};\infty)<\infty.$
\end{lemma}
\begin{proof}
We note that
\[\sum_{n\le x}\frac{\operatorname{Re}f(n)\overline{f(n+1)})}{n}\ge \varepsilon \log x.\]
We can now apply Lemma~\ref{generaltao}, since the only fact that was used in the proof is  that logarithmic correlation is large to conclude the result.  
\end{proof}
\begin{remark} The conclusion of the lemma also holds if $f:\mathbb{N}\to\mathbb{T}$
satisfies
\[\sum_{n\le x}\frac{|\triangle f(n)|^2}{n}\ge 2(1+\varepsilon)\log x\]
for some $\varepsilon>0.$ In other words, if $\sum_{n\le x}\frac{|\triangle f(n)|^2}{n}$ is bounded away from $2\log x,$ then $\mathbb{D}(f(n),\chi_1(n)n^{it_{\chi_1}};\infty)<\infty.$
\end{remark}
\begin{prop}\label{pretentiouskatai}
Let $f:\mathbb{N}\to \mathbb{T}$ be a multiplicative function and $\mathbb{D}(f,n^{it}\chi(n);\infty)<\infty$ for some $t\in\mathbb{R}$ and a primitive character $\chi$ of conductor $q.$ Then
\[\sum_{n\le x}\frac{|\triangle f(n)|^2}{n}=2(1-E(f)+o(1))\log x\]
where 
\[E(f)=\frac{\mu(q)}{q}\prod_{\substack{p\ge 1\\ p\nmid q}}\left(2\operatorname{Re}\left(1-\frac{1}{p}\right)\left(\sum_{k\ge 0}\frac{f(p^k)\overline{\chi(p^k)}p^{-ikt}}{p^k}\right)-1\right).\] 
\end{prop}
\begin{proof}
Applying Corollary~\ref{keytotao}
we have that 
\[M(y)=\sum_{n\le y}f(n)\overline{f(n+1)})=y\frac{\mu(q)}{q}\prod_{\substack{p\ge 1\\ p\nmid q}}\left(2\operatorname{Re}\left(1-\frac{1}{p}\right)\left(\sum_{k\ge 0}\frac{f(p^k)\overline{\chi(p^k)}p^{-ikt}}{p^k}\right)-1\right)+o(y).\]
Consequently,
\[\sum_{n\le x}\frac{\operatorname{Re}f(n)\overline{f(n+1)})}{n}=\frac{M(x)}{x}+\int_{1}^x\frac{M(y)}{y^2}dy=\log x \cdot E(f)+o(\log x)\]
and
\[\sum_{n\le x}\frac{|\triangle f(n)|^2}{n}=2\log x -2\sum_{n\le x}\frac{\operatorname{Re}f(n)\overline{f(n+1)})}{n}+O(1)=2(1-E(f)+o(1))\log x.\]
\end{proof}
\begin{corollary}\label{largecor}
Let $f:\mathbb{N}\to \mathbb{T}$ be a multiplicative function such that $\mathbb{D}(f,n^{it}\chi(n);\infty)<\infty$ for some $t\in\mathbb{R}$ and primitive character $\chi$ of conductor $q.$ Suppose that 
 \[\sum_{n\le x}\frac{|\triangle f(n)|^2}{n}=o(\log x).\]
 Then, $f(n)=n^{it}.$
\end{corollary}
\begin{proof}
Proposition~\ref{pretentiouskatai}  implies that  $1=E(f).$ Since
for all $p\ge 1$ each Euler factor
\[2\left(1-\frac{1}{p}\right)\sum_{k\ge 0}\frac{\operatorname{Re}f(p^k)\overline{\chi(p^k)}p^{-ikt}}{p^k}-1\le 1.\]
we must therefore have $q=1$ and 
\[2\left(1-\frac{1}{p}\right)\sum_{k\ge 0}\frac{\operatorname{Re}f(p^k)p^{-ikt}}{p^k}-1=1.\]
This is possible if only if $f(p^k)=p^{kit}$ for all $p\ge 1$ and $k\ge 1.$ The result follows.
\end{proof}
Theorem~\ref{introkatai} now follows from the following
\begin{prop}
Let $f:\mathbb{N}\to\mathbb{T}$ be a multiplicative function such that \[
\sum_{n\le x}\frac{|\triangle f(n)|^2}{n}=o(\log x).
\]
Then $f(n)=n^{it}$ for some $t\in\mathbb{R}.$
\end{prop}
\begin{proof}

Applying Lemma~\ref{kataireduction} we can find a primitive character $\chi$ and $t\in\mathbb{R}$ such that $\mathbb{D}(f(n),\chi(n)n^{it};\infty)<\infty.$ We now apply Corollary~\ref{largecor} to conclude that $f(n)=n^{it}.$
\end{proof}

\end{section}
\begin{section}{Applications to the binary additive problems}
As was mentioned in the introduction Br{\"u}dern established the following result.

\begin{theoremn}[\ref{brudern}]{\bf{[Br{\"u}dern, 2008]}}
Suppose $A$ and $B$ are multiplicative sequences of positive density $\rho_A$ and $\rho_B$ respectively. For $k\ge 1,$ let 
\[a(p^k)=\rho_A(p^k)/p^k-\rho_A(p^{k-1})/p^{k-1}\]
Define $b(h)$ in the same fashion. Then, $r(n)=\rho_A\rho_B\sigma(n)n+o(n)$
where
\[\sigma(n)=\prod_{p^m\vert| n}\left(1+\sum_{k=1}^m\frac{p^{k-1}a(p^k)b(p^k)}{p-1}-\frac{p^{m}a(p^{m+1})b(p^{m+1})}{(p-1)^2}\right).\]
\end{theoremn}

We now sketch how one can derive this from our main result. 
\begin{proof}
Let $f(n)=\text{I}_A(n)$ and $g(n)=\text{I}_B(n).$ Clearly both, $f$ and $g$ are multiplicative  taking values $\{0,1\}.$ Since $\rho_A>0,$ we have 
\[\limsup_{x} \frac{1}{x}\sum_{n\le x}f(n)>0.\]
Theorem of Delange readily implies that $\mathbb{D}(1,f;\infty)<\infty.$ By analogy, $\mathbb{D}(1,g;\infty)<\infty.$ Furthermore,
\[\rho_A=\lim_{x\to\infty} \frac{1}{x}\sum_{n\le x}f(n)=\textfrak{P}(f,1,\infty)\]
and
\[\rho_B=\lim_{x\to\infty} \frac{1}{x}\sum_{n\le x}g(n)=\textfrak{P}(g,1,\infty).\]
 Notice that
\[r(n)=\sum_{m\le n}f(m)g(n-m).\]
We note that combining the the proof of Corollary~\ref{introlinear} we may let $a=1,$ $c=0,$ $b=n,$ $d=-1$ in Corollary~\ref{introlinear}. Despite the fact that $d=n\to\infty$ the error term is still bounded by~\eqref{error}. Corollary~\ref{introlinear} gives
\[r(n)=\sum_{d\vert n}\frac{G(f;g;d;\infty)}{d}n+o(n).\]
A straightforward manipulation with Euler factors show that the latter has the Euler product described above.
\end{proof}
\begin{remark}
In case one of the sets $A,B$ has density zero, say $\rho_A=0$ we can apply Delange's theorem to conclude
\[r(n)=\sum_{m\le n }f(m)g(n-m)\le \sum_{m\le n }f(m)=o(n).\]
\end{remark}
\end{section}
\nocite{MR1618321}
\bibliographystyle{alpha}
\bibliography{Markovbib}

\begin{thebibliography}{WTS96}

\bibitem[Br{\"u}09]{MR2508639}
J{\"o}rg Br{\"u}dern.
\newblock Binary additive problems and the circle method, multiplicative
  sequences and convergent sieves.
\newblock In {\em Analytic number theory}, pages 91--132. Cambridge Univ.
  Press, Cambridge, 2009.

\bibitem[Ell92]{MR1292619}
P.~D. T.~A. Elliott.
\newblock On the correlation of multiplicative functions.
\newblock {\em Notas Soc. Mat. Chile}, 11(1):1--11, 1992.

\bibitem[Ell10]{MR2658182}
P.~D. T.~A. Elliott.
\newblock The value distribution of additive arithmetic functions on a line.
\newblock {\em J. Reine Angew. Math.}, 642:57--108, 2010.

\bibitem[Erd46]{MR0015424}
P.~Erd{\"o}s.
\newblock On the distribution function of additive functions.
\newblock {\em Ann. of Math. (2)}, 47:1--20, 1946.

\bibitem[Erd57]{MR0098702}
Paul Erd{\H{o}}s.
\newblock Some unsolved problems.
\newblock {\em Michigan Math. J.}, 4:291--300, 1957.

\bibitem[Erd85a]{MR797781}
P.~Erd{\H{o}}s.
\newblock On some of my problems in number theory {I} would most like to see
  solved.
\newblock In {\em Number theory ({O}otacamund, 1984)}, volume 1122 of {\em
  Lecture Notes in Math.}, pages 74--84. Springer, Berlin, 1985.

\bibitem[Erd85b]{MR851041}
P.~Erd{\H{o}}s.
\newblock Some applications of probability methods to number theory.
\newblock In {\em Mathematical statistics and applications, {V}ol.\ {B} ({B}ad
  {T}atzmannsdorf, 1983)}, pages 1--18. Reidel, Dordrecht, 1985.

\bibitem[GS07a]{MR2276774}
Andrew Granville and K.~Soundararajan.
\newblock Large character sums: pretentious characters and the
  {P}\'olya-{V}inogradov theorem.
\newblock {\em J. Amer. Math. Soc.}, 20(2):357--384, 2007.

\bibitem[GS07b]{MR2290492}
Andrew Granville and K.~Soundararajan.
\newblock Sieving and the {E}rd{\H o}s-{K}ac theorem.
\newblock In {\em Equidistribution in number theory, an introduction}, volume
  237 of {\em NATO Sci. Ser. II Math. Phys. Chem.}, pages 15--27. Springer,
  Dordrecht, 2007.

\bibitem[Hal71]{MR0319930}
G.~Hal{\'a}sz.
\newblock On the distribution of additive and the mean values of multiplicative
  arithmetic functions.
\newblock {\em Studia Sci. Math. Hungar.}, 6:211--233, 1971.

\bibitem[Hal75]{MR0369292}
G{\'a}bor Hal{\'a}sz.
\newblock On the distribution of additive arithmetic functions.
\newblock {\em Acta Arith.}, 27:143--152, 1975.
\newblock Collection of articles in memory of Juri{\u\i} Vladimirovi{\v{c}}
  Linnik.

\bibitem[Hil88a]{MR965752}
Adolf Hildebrand.
\newblock An {E}rd{\H o}s-{W}intner theorem for differences of additive
  functions.
\newblock {\em Trans. Amer. Math. Soc.}, 310(1):257--276, 1988.

\bibitem[Hil88b]{MR932664}
Adolf Hildebrand.
\newblock Multiplicative functions at consecutive integers. {II}.
\newblock {\em Math. Proc. Cambridge Philos. Soc.}, 103(3):389--398, 1988.

\bibitem[K{\'a}t70]{MR0250991}
I.~K{\'a}tai.
\newblock On a problem of {P}. {E}rd{\H o}s.
\newblock {\em J. Number Theory}, 2:1--6, 1970.

\bibitem[K{\'a}t83]{MR729295}
I.~K{\'a}tai.
\newblock Some problems in number theory.
\newblock {\em Studia Sci. Math. Hungar.}, 16(3-4):289--295, 1983.

\bibitem[K{\'a}t91]{MR1153489}
I.~K{\'a}tai.
\newblock Multiplicative functions with regularity properties. {VI}.
\newblock {\em Acta Math. Hungar.}, 58(3-4):343--350, 1991.

\bibitem[K{\'a}t00]{MR1764803}
I.~K{\'a}tai.
\newblock Continuous homomorphisms as arithmetical functions, and sets of
  uniqueness.
\newblock In {\em Number theory}, Trends Math., pages 183--200. Birkh\"auser,
  Basel, 2000.

\bibitem[KMT]{MR6}
Maksym~Radziwi{{\l}}{{\l}} Kaisa~Matom{{\"a}}ki and Terrence Tao.
\newblock An averaged form of {C}howla's conjecture.
\newblock {\em to appear in Algebra and Number Theory}, arxiv.

\bibitem[MM80]{MR603061}
J.-L. Mauclaire and Leo Murata.
\newblock On the regularity of arithmetic multiplicative functions. {I}.
\newblock {\em Proc. Japan Acad. Ser. A Math. Sci.}, 56(9):438--440, 1980.

\bibitem[MR]{MR5}
Kaisa Matom{{\"a}}ki and Maksym Radziwi{{\l}}l.
\newblock Multiplicative functions in short intervals.
\newblock {\em to appear in Annals of Mathematics}, arxiv.

\bibitem[NT98]{MR1618321}
Mohan Nair and G{\'e}rald Tenenbaum.
\newblock Short sums of certain arithmetic functions.
\newblock {\em Acta Math.}, 180(1):119--144, 1998.

\bibitem[Pho00]{MR1798718}
Bui~Minh Phong.
\newblock A characterization of some unimodular multiplicative functions.
\newblock {\em Publ. Math. Debrecen}, 57(3-4):339--366, 2000.

\bibitem[Pho14]{MR3158863}
B.~M. Phong.
\newblock Additive functions at consecutive integers.
\newblock {\em Acta Math. Hungar.}, 142(1):260--274, 2014.

\bibitem[Ste02]{MR1964870}
Gediminas Stepanauskas.
\newblock The mean values of multiplicative functions. {V}.
\newblock In {\em Analytic and probabilistic methods in number theory
  ({P}alanga, 2001)}, pages 272--281. TEV, Vilnius, 2002.

\bibitem[Taoa]{MR3}
Terrrence Tao.
\newblock The {E}rd{\H{o}}s discrepancy problem.
\newblock {\em to appear in Discrete Analysis}, arxiv.

\bibitem[Taob]{MR4}
Terrrence Tao.
\newblock The logarithmically averaged {C}howla and {E}lliot conjectures for
  two-point correlations.
\newblock arxiv.

\bibitem[WTS96]{MR1373561}
Eduard Wirsing, Yuan-Sheng Tang, and Pin-tsung Shao.
\newblock On a conjecture of {K}\'atai for additive functions.
\newblock {\em J. Number Theory}, 56(2):391--395, 1996.

\bibitem[WZ01]{MR1864627}
Eduard Wirsing and Don Zagier.
\newblock Multiplicative functions with difference tending to zero.
\newblock {\em Acta Arith.}, 100(1):75--78, 2001.

\end{thebibliography}
\end{document}